\documentclass[a4paper,11pt]{article}
\usepackage[utf8]{inputenc}
\usepackage{a4wide}
\usepackage{latexsym,amsfonts,amsmath,amssymb,mathrsfs,url,amsthm}
\usepackage{mathtools}
\usepackage{color,graphicx}
\usepackage{lipsum}
\usepackage{hyperref}
\usepackage{xcolor}
\newtheorem{theorem}{Theorem}[section]
\newtheorem{lemma}[theorem]{Lemma}

\newtheorem{algorithm}[theorem]{Algorithm}
\newtheorem{remark}[theorem]{Remark}

\newtheorem{definition}[theorem]{Definition}
\newtheorem{corollary}[theorem]{Corollary}
\providecommand{\keywords}[1]
{
  \noindent \small	
  \textbf{Keywords:} #1
}
\providecommand{\amscode}[1]
{
  \noindent \small	
  \textbf{AMS subject classifications:} #1
}
\newcommand{\rd}{\, \mathrm{d}}
\newcommand{\rand}{\mathrm{rand}}
\newcommand{\wor}{\mathrm{wor}}
\newcommand{\rms}{\mathrm{rms}}
\newcommand{\app}{\mathrm{app}}
\newcommand{\bszero}{\boldsymbol{0}}
\newcommand{\bsh}{\boldsymbol{h}}

\newcommand{\bsl}{\boldsymbol{\ell}}
\newcommand{\bsx}{\boldsymbol{x}}
\newcommand{\bsy}{\boldsymbol{y}}
\newcommand{\bsz}{\boldsymbol{z}}
\newcommand{\bsgamma}{\boldsymbol{\gamma}}
\newcommand{\bsDelta}{\boldsymbol{\Delta}}
\newcommand{\EE}{\mathbb{E}}
\newcommand{\II}{\mathbb{I}}
\newcommand{\NN}{\mathbb{N}}

\newcommand{\ZZ}{\mathbb{Z}}
\newcommand{\Acal}{\mathcal{A}}
\newcommand{\Pcal}{\mathcal{P}}
\newcommand{\Zcal}{\mathcal{Z}}
\newcommand{\Mod}[1]{\ (\mathrm{mod}\ #1)}

\allowdisplaybreaks

\title{$L_2$-approximation using randomized lattice algorithms}
\author{Mou Cai\thanks{Graduate School of Engineering, The University of Tokyo, 7-3-1 Hongo, Bunkyo-ku, Tokyo 113-8656, Japan ({\tt caimoumou@g.ecc.u-tokyo.ac.jp}; {\tt goda@frcer.t.u-tokyo.ac.jp})}, 
Takashi Goda\footnotemark[1], 
Yoshihito Kazashi\thanks{Department of Mathematics, University of Manchester, Oxford Road, Manchester M13 9PL, UK ({\tt y.kazashi@manchester.ac.uk})}}
\date{\today}

\begin{document}

\maketitle

\sloppy

\begin{abstract}
We propose a randomized lattice algorithm for approximating multivariate periodic functions over the $d$-dimensional unit cube from the weighted Korobov space with mixed smoothness $\alpha > 1/2$ and product weights $\gamma_1,\gamma_2,\ldots\in [0,1]$. Building upon the deterministic lattice algorithm by Kuo, Sloan, and Wo\'{z}niakowski (2006), we incorporate a randomized quadrature rule by Dick, Goda, and Suzuki (2022) to accelerate the convergence rate. This randomization involves drawing the number of points for function evaluations randomly, and selecting a good generating vector for rank-1 lattice points using the randomized component-by-component algorithm. We prove that our randomized algorithm achieves a worst-case root mean squared $L_2$-approximation error of order $M^{-\alpha(2\alpha+1)/(4\alpha+1)+\varepsilon}$ for an arbitrarily small $\varepsilon > 0$, where $M$ denotes the maximum number of function evaluations, and that the error bound is independent of the dimension $d$ if the weights satisfy $\sum_{j=1}^\infty \gamma_j^{1/\alpha} < \infty$. Our upper bound converges faster than a lower bound on the worst-case $L_2$-approximation error for deterministic rank-1 lattice-based approximation proved by Byrenheid, K\"{a}mmerer, Ullrich, and Volkmer (2017). We also show a lower error bound of order $M^{-\alpha/2-1/2}$ for our randomized algorithm, leaving a slight gap between the upper and lower bounds open for future research.
\end{abstract}


\amscode{41A63, 42B05, 65C05, 65D15, 65T40}

\section{Introduction}\label{sec:intro}
In this paper, we study a lattice-based algorithm for multivariate $L_2$-approximation of periodic functions defined over the $d$-dimensional unit cube. We assume that a target function $f$ belongs to the weighted Korobov space, denoted by $H_{d, \alpha, \bsgamma}$, with mixed smoothness $\alpha > 1/2$ and product weights $\bsgamma=(\gamma_1,\gamma_2,\ldots)$ where $0\leq \gamma_j\leq 1$ for all $j$, whose precise definition will be given in the next section. This allows for a pointwise representation of $f$ by the absolutely convergent Fourier series
\begin{align*}
f(\bsx)=\sum_{\bsh\in \ZZ^d} \hat{f}(\bsh) \exp(2\pi i \bsh \cdot \bsx),\quad \text{for any } \bsx\in[0,1)^d,
\end{align*}
where $\hat{f}(\bsh)$ denotes the $\bsh$-th Fourier coefficient, defined as 
\[ \hat{f}(\bsh)=\int_{[0,1)^d}f(\bsx)\exp(-2\pi i \bsh\cdot\bsx)\rd \bsx, \]
and $\bsh \cdot \bsx = h_1 x_1 + \cdots + h_d x_d$ denotes the usual dot product. The lattice-based algorithm for multivariate approximation has been studied extensively, as seen in \cite{byrenheid2017tight,cools2020lattice,cools2021fast,kaarnioja2022fast,kammerer2015approximation,kammerer2019approximation,kuo2021function,kuo2024constructing,kuo2006lattice,zeng2009spline,zeng2006error}. For an overview of these works, we also refer to \cite[Chapters~13--15]{dick2022lattice}.

We start by introducing a deterministic algorithm based on the work of Kuo, Sloan, and Wo\'{z}niakowski \cite{kuo2006lattice}, which our approach builds upon.
We first consider approximating a function $f$ by truncating the Fourier series to a finite index set $\Acal_d \subset \ZZ^d$. 
Each of the Fourier coefficients within the index set $\Acal_d$ is then approximated by a rank-1 lattice rule with $N$ points and a generating vector $\bsz \in \{1,\ldots,N-1\}^d$. Here, the rank-1 lattice rule is an equal-weight quadrature rule for approximating high-dimensional integrals:
\[ \int_{[0,1]^d}f(\bsx)\rd \bsx \approx \frac{1}{N}\sum_{k=0}^{N-1}f\left(\left\{\frac{k \bsz}{N}\right\}\right),\]
where $\{x\}=x-\lfloor x\rfloor $ denotes the fractional part of a non-negative real number $x$ and is applied component-wise to a vector. 
By applying this rule to approximate all $\hat{f}(\bsh)$ for $\bsh \in \Acal_d$, the resulting approximation is given by
\begin{align}\label{eq:lattice_algorithm}
A_{N,\bsz,\Acal_d}(f)(\bsx) := \sum_{\bsh \in \Acal_d} \hat{f}_{N,\bsz}(\bsh) \exp(2\pi i \bsh \cdot \bsx),
\end{align}
where
\begin{align*}
\hat{f}_{N,\bsz}(\bsh) := \frac{1}{N} \sum_{k=0}^{N-1} f\left(\left\{\frac{k \bsz}{N}\right\}\right) \exp\left(-2\pi i k\bsh \cdot \bsz/N\right).
\end{align*}

With this approach, choosing a good generating vector $\bsz$ is obviously important. One of the major quality measures of a rank-1 lattice point set (or the corresponding generating vector $\bsz$) is the worst-case $L_2$-approximation error:
\begin{align}\label{eq:worst_error}
e^{\wor\text{-}L_2\text{-}\app}_{d, \alpha, \bsgamma}(A_{N,\bsz,\Acal_d}):=\sup_{\substack{ f\in H_{d, \alpha, \bsgamma} \\ \|f\|_{d, \alpha, \bsgamma} \leq 1}} \| f-A_{N,\bsz,\Acal_d}(f)\|_{L_2},
\end{align}
where $\|\cdot \|_{d, \alpha, \bsgamma}$ denotes the norm of the weighted Korobov space $H_{d, \alpha, \bsgamma}$ and $\|\cdot\|_{L_2}$ denotes the $L_2$ norm over $[0,1)^d$. 

With a good choice of $\bsz$, a worst-case $L_2$-approximation error bound of order $N^{-\alpha/2}$ has been shown \cite{cools2020lattice, cools2021fast, kuo2006lattice}. The convergence rate obtained is tight, due to the lower bound of order $N^{-\alpha/2}$ for any lattice-based algorithm---i.e., any method that uses function evaluations at rank-1 lattice points---proven by Byrenheid, K\"{a}mmerer, Ullrich, and Volkmer \cite{byrenheid2017tight}. This rate is inferior to the optimal rate $N^{-\alpha}(\log N)^{(d-1)\alpha}$ of the best approximation with $N$ function evaluations \cite[Theorem~6.10]{byrenheid2016sampling}. Nonetheless, lattice-based approximation methods attract significant attention because the lattice structure enables efficient computation of the approximate function via the fast Fourier transform \cite{kaarnioja2022fast,kammerer2015approximation}. Additionally, the implied constant of the error bound depends only polynomially on, or is even independent of, the dimension under certain conditions on the weights $\bsgamma=(\gamma_1,\gamma_2,\ldots)$ \cite{cools2020lattice}.

In this paper, we consider a randomized counterpart of the lattice-based algorithm of the form \eqref{eq:lattice_algorithm}. 
It turns out that our algorithm achieves faster convergence than any deterministic lattice-based algorithm in the sense we detail below. 
These results echo similar findings on the improved convergence rates made possible by randomized algorithms \cite{dick2011higher,dick2022component,goda2015construction,goda2024randomizing,kritzer2019lattice,novak1988deterministic,novak2008tractability,novak2010tractability,owen1997scrambled,ullrich2017monte}. 
In our proposed algorithm, we choose the number of points $N$ and the generating vector $\bsz$ randomly (with an additional random shift $\bsDelta$) rather than fixing them. 
This results in a \emph{randomized lattice-based algorithm} denoted by $A^{\rand}_{N,\bsz,\bsDelta,\Acal_d}$, whose precise definition will be introduced later in this paper, for multivariate $L_2$-approximation. As we are working in the randomized setting, the (worst-case) root-mean-square error (RMSE)
\begin{align}\label{eq:rms_error}
e^{\rms\text{-}L_2\text{-}\app}_{d, \alpha, \bsgamma}(A^{\rand}_{N,\bsz,\bsDelta,\Acal_d}):=\sup_{\substack{ f\in H_{d, \alpha, \bsgamma} \\ \|f\|_{d, \alpha, \bsgamma} \leq 1}} \sqrt{\EE\left[\| f-A^{\rand}_{N,\bsz,\bsDelta,\Acal_d}(f)\|^2_{L_2}\right] }
\end{align}
is used as a quality measure.
Note that, due to the randomness in our algorithm $A^{\rand}_{N,\bsz,\bsDelta,\Acal_d}$, we take the expectation of the squared $L_2$ error in \eqref{eq:rms_error}.

Our approach is motivated by recent advances in randomized lattice rules for numerical integration in weighted Korobov spaces. 
In \cite{kritzer2019lattice}, Kritzer, Kuo, Nuyens, and Ullrich revisited the idea of randomly drawing $N$ due to Bakhvalov \cite{bahvalov1961estimate}, and proved the following result: a randomized lattice rule, which draws the number of points $N$ randomly from a set of primes in the interval $(\lceil M/2\rceil, M]$ for a given $M$ and randomly selects a generating vector from a set of ``good'' ones, achieves the worst-case randomized integration error of order $M^{-\alpha - 1/2 + \varepsilon}$ for an arbitrarily small $\varepsilon > 0$. This improves upon the deterministic worst-case integration error of good rank-1 lattice rules, which is of order $N^{-\alpha + \varepsilon}$. However, it is challenging to randomly draw a generating vector from the uniform distribution over a set of ``good'' ones. To address this, Dick, Goda, and Suzuki introduced a randomized component-by-component (CBC) algorithm in \cite{dick2022component}, where the deterministic worst-case integration error is used as a quality measure of generating vectors. Further progress on randomized lattice rules can be found in \cite{goda2024randomized, kuo2023random, nuyens2024randomised}.

Following \cite{dick2022component}, in our randomized lattice-based algorithm for multivariate $L_2$-approximation, we draw the number of points $N$ randomly from a set of primes in the interval $(\lceil M/2\rceil, M]$ for a given $M$ and then select a generating vector using a randomized CBC algorithm. 
We prove that our randomized lattice-based algorithm attains the worst-case RMSE of order $M^{-\alpha(2\alpha+1)/(4\alpha+1)+\varepsilon}$, and that the error bound is independent of the dimension $d$ if the weight parameters satisfy $\sum_{j=1}^\infty \gamma_j^{1/\alpha} < \infty$. It is striking that the convergence rate we obtain is better than the lower bound for any deterministic lattice-based algorithm as proven by Byrenheid, K\"{a}mmerer, Ullrich, and Volkmer \cite{byrenheid2017tight}. Furthermore, we prove a lower error bound of order $M^{-\alpha/2 - 1/2}$ for our randomized algorithm. Filling the gap between the upper and lower bounds is left open for future research.

While the results in this paper shows that randomization improves the approximation rate for lattice-based algorithms, it is worth noting that in general randomization does not improve the optimal approximation rate. Indeed, Novak \cite{novak1992optimal} proved that the best possible rate for the randomized error coincides with that for the (deterministic) worst-case error in general separable Hilbert spaces for algorithms using (either random or deterministic) linear functionals as information about the target function. Later, as a corollary of general results, Krieg and Ullrich \cite{krieg2021function} proved that, for the Korobov spaces, the best possible convergence rate of $L_2$-approximation is the same whether the algorithms use linear functionals or function evaluations as information, up to a logarithmic factor. By combining these results, we conclude that the same optimal randomized error rate holds for randomized algorithms using only $N$ function values, up to a logarithmic factor. By contrast, the results in this paper show that, when restricted to the class of lattice-based algorithms, randomization does improve $L_2$-approximation.

The rest of this paper is organized as follows. In the next section, we present the necessary preliminaries, such as weighted Korobov spaces and rank-1 lattice rules. In Section~\ref{sec:construction}, we introduce our randomized CBC algorithm for multivariate $L_2$-approximation in weighted Korobov spaces and prove an upper bound on the worst-case $L_2$-approximation error that holds for any realization generated by our CBC algorithm. In Section~\ref{sec:upper}, we present our randomized lattice-based algorithm and give a theoretical analysis for its worst-case RMSE. In Section~\ref{sec:lower}, we discuss the corresponding lower bounds for our randomized lattice-based algorithm. 
We conclude this paper with the numerical experiments presented in Section~\ref{sec:numerics}.

\section{Preliminaries}\label{sec:pre}
\subsection{Lattice rules}
Throughout this paper, we denote the set of positive integers by $\NN$. 
We first introduce the rank-1 lattice point set.
\begin{definition}[Rank-1 lattice point set]
    For $N\in \NN$ with $N\geq 2$, let $\bsz=(z_1,\ldots,z_d)\in \{1,\ldots,N-1\}^d$ be given. The rank-1 lattice point set $P_{N,\bsz}$ is defined by
    \[ P_{N,\bsz}:=\left\{ \bsx_n=\left(\left\{ \frac{nz_1}{N}\right\},\ldots,\left\{ \frac{nz_d}{N}\right\}\right) \,\mid \, 0\leq n<N\right\}.\]
\end{definition}

A quadrature that uses a rank-1 lattice point set to approximate the integral of a function $f$ as
\[ \frac{1}{N}\sum_{\bsx\in P_{N,\bsz}}f(\bsx)\approx \int_{[0,1)^d}f(\bsx)\rd \bsx \]
is called a \emph{rank-1 lattice rule} with generating vector $\bsz$.

The dual lattice is important for obtaining a good lattice point set.
\begin{definition}[Dual lattice]
    For $N\in \NN$ with $N\geq 2$ and a generating vector $\bsz\in \{1,\ldots,N-1\}^d$, the set
    \[ P_{N, \bsz}^{\perp}:=\left\{\bsh \in \ZZ^{d} \mid \bsh \cdot \bsz \equiv 0 \pmod N\right\} \]
    is called the dual lattice of the rank-1 lattice point set $P_{N, z}$.
\end{definition} 

The following property of rank-1 lattice rules holds, implying that a rank-1 lattice rule is exact for integrating the Fourier modes if the frequency $\bsh$ is not in the set $P_{N, \bsz}^{\perp}\setminus \{\bszero\}$. 
We refer to \cite[Lemma~1.9]{dick2022lattice} for the proof.

\begin{lemma}[Character property]\label{lem:character}
    For $N\in \NN$ with $N\geq 2$ and a generating vector $\bsz\in \{1,\ldots,N-1\}^d$, 
    \[ \frac{1}{N}\sum_{n=0}^{N-1} \exp \left(2 \pi i \bsh \cdot \bsx_n\right)= \begin{cases}1 & \text{if $\bsh \in P_{N, \bsz}^{\perp}$,} \\ 0 & \text{otherwise,}\end{cases} \]
    holds for any vector $\bsh \in \ZZ^{d}$. 
\end{lemma} 

\subsection{Weighted Korobov spaces}
Let $\alpha>1/2$ be a real number and $\bsgamma=(\gamma_1,\gamma_2,\ldots)$ be a sequence of non-negative real numbers, where we assume that $0\leq \gamma_j\leq 1$ for all $j$.
Although the condition $\gamma_j \leq 1$ can be omitted without altering the essential nature of the results presented in this paper, we assume this condition throughout to avoid overly technical arguments. We will explicitly indicate where this assumption is used along the way.
For $\bsh \in \ZZ^{d}$, define
\begin{align*}
r_{\alpha, \bsgamma}(\bsh):=\prod_{\substack{j=1 \\ h_{j} \neq 0}}^{d} \frac{\left|h_{j}\right|^{\alpha}}{\gamma_{j}},
\end{align*}
where the empty product is set to $1$, i.e., $r_{\alpha, \bsgamma}(\bszero)=1$. 
If there exists an index $j\in \{1,\ldots,d\}$ such that $\gamma_j=0$ and $h_j\neq 0$, we set $r_{\alpha, \bsgamma}(\bsh)=\infty$. 
Given the assumption that $\gamma_j \leq 1$ for all $j$, it follows that $r_{\alpha, \bsgamma}(\bsh) \geq 1$ for any $\bsh \in \ZZ^{d}$.
Then the weighted Korobov space, denoted by $H_{d, \alpha, \bsgamma}$, is a reproducing kernel Hilbert space with the reproducing kernel
\[ K_{d, \alpha, \bsgamma}(\bsx, \bsy)=\sum_{\bsh \in \ZZ^{d}} \frac{\exp (2 \pi i \bsh \cdot(\bsx-\bsy))}{\left(r_{\alpha, \bsgamma}(\bsh)\right)^{2}}, \]
and the inner product
\[ \langle f, g\rangle_{d, \alpha, \bsgamma}=\sum_{\bsh \in \ZZ^{d}}\left(r_{\alpha, \bsgamma}(\bsh)\right)^{2} \hat{f}(\bsh) \overline{\hat{g}(\bsh)}. \] 
We denote the induced norm by $\|f\|_{d, \alpha, \bsgamma}:=\sqrt{\langle f, f\rangle_{d, \alpha, \bsgamma}}$.

Here the parameter $\alpha > 1/2$ measures the smoothness of periodic functions. The sequence of non-negative weights $\gamma_{1}, \gamma_{2}, \ldots$ plays a role in moderating the relative importance of different variables \cite{sloan1998when}. Here if $\gamma_{j}=0$ for some $1 \leq j \leq d$, we assume that all the Fourier coefficients $\hat{f}(\bsh)$ and $\hat{g}(\bsh)$ for $\bsh \in \ZZ^{d}$ such that $h_{j} \neq 0$ are 0 and we set $\infty \cdot 0=0$. When $\alpha$ is an integer, it is directly related to the number of available square-integrable partial mixed derivatives in each variable \cite[Section~2.1]{dick2022lattice}.

\subsection{Lattice algorithm for approximation}
Although this paper is concerned with a randomized lattice-based algorithm for function approximation, we first explain a lattice-based algorithm in the deterministic setting in more detail. Let $\Acal_d\subset \ZZ^d$ be a finite index set. As seen in the previous section, we approximate $f\in H_{d, \alpha, \bsgamma}$ as
\begin{align*}
    f(\bsx) & = \sum_{\bsh\in \ZZ^d} \hat{f}(\bsh) \exp(2\pi i \bsh \cdot \bsx)\\
    & \approx \sum_{\bsh\in \Acal_d} \hat{f}(\bsh) \exp(2\pi i \bsh \cdot \bsx)\\
    & \approx \sum_{\bsh \in \Acal_d} \hat{f}_{N,\bsz}(\bsh) \exp(2\pi i \bsh \cdot \bsx) =: A_{N,\bsz,\Acal_d}(f)(\bsx),
\end{align*}
where 
\begin{align*}
\hat{f}_{N,\bsz}(\bsh) := \frac{1}{N} \sum_{k=0}^{N-1} f\left(\left\{\frac{k \bsz}{N}\right\}\right) \exp\left(-2\pi i k\bsh \cdot \bsz/N\right).
\end{align*}
That is, we first truncate the whole Fourier series of $f$ to the index set $\Acal_d$ and then approximate all the Fourier coefficients $\hat{f}(\bsh)$ with $\bsh\in \Acal_d$ by a rank-1 lattice rule. 

When selecting the index set $\Acal_d$, it is desirable to minimize its size to reduce computational cost while including as many indices as possible that correspond to ``large'' Fourier coefficients.
To this end, for a real number $T > 0$, we define $\Acal_{d}(T)$ by
\begin{align*}
\Acal_d(T) := \left\{ \bsh \in \ZZ^d : (r_{\alpha, \bsgamma}(\bsh))^2 \leq T \right\},
\end{align*}
and approximate $f \in H_{d,\alpha,\bsgamma}$ by $A_{N,\bsz,\Acal_d(T)}(f)(\bsx)$. We choose $T$ suitably so that the approximation error is small depending on $N$.

The worst-case error of the deterministic approximation algorithm $A_{N,\bsz,\Acal_d(T)}$ in the space $H_{d,\alpha,\bsgamma}$ is defined by 
\begin{align*}
e^{\wor\text{-}L_2\text{-}\app}_{d,\alpha,\bsgamma}(A_{N,\bsz,\Acal_d(T)}) := 
\sup_{\substack{f \in H_{d,\alpha,\bsgamma} \\ \| f \|_{d,\alpha,\bsgamma} \leq 1}} 
\left\| f - A_{N,\bsz,\Acal_d(T)}(f) \right\|_{L_2}.
\end{align*}
As a reference value, we use the initial error where we approximate the function with $0$, i.e.,
\begin{align*}
e^{\wor\text{-}L_2\text{-}\app}_{d,\alpha,\bsgamma}(0) := 
\sup_{\substack{f \in H_{d,\alpha,\bsgamma} \\ \| f \|_{d,\alpha,\bsgamma} \leq 1}} 
\left\| f  \right\|_{L_2}=1.
\end{align*}
Here the last equality holds since we have 
\[ \|f\|^2_{L_2}=\sum_{\bsh \in \ZZ^{d}} |\hat{f}(\bsh)|^2 \leq \sum_{\bsh \in \ZZ^{d}} \left(r_{\alpha, \bsgamma}(\bsh)\right)^{2}|\hat{f}(\bsh)|^2 = \| f \|^2_{d,\alpha,\bsgamma}\]
for any $f\in H_{d,\alpha,\bsgamma}$, where the inequality follows from $r_{\alpha, \bsgamma}(\bsh) \geq 1$ for all $\bsh$ under the assumption that $\alpha>1/2$ and $0\leq \gamma_j\leq 1$ for all $j$, and for the constant function $f=1$, it holds that $\|f\|_{L_2}= \| f \|_{d,\alpha,\bsgamma}=1$, see \cite{kuo2006lattice}.

As shown, for instance, in \cite[Eq.~(13.12)]{dick2022lattice}, it has been known that the following worst-case error bound holds:
\begin{lemma}
Let $\alpha > 1/2$ and $\bsgamma = (\gamma_1, \gamma_2, \ldots) \in [0, 1]^{\mathbb{N}}$. For any $N\in \NN$ with $N\geq 2$, $\bsz \in \{1, \ldots, N-1\}^d$ and $T>0$, the squared worst-case $L_2$-approximation error of the lattice-based algorithm $A_{N,\bsz,\Acal_d(T)}$ is bounded above by
\begin{align*}
\left(e^{\wor\text{-}L_2\text{-}\app}_{d,\alpha,\bsgamma}(A_{N,\bsz,\Acal_d(T)})\right)^2 \leq \frac{1}{T}+T \left[R_{N,d,\alpha, \bsgamma}(\bsz)\right]^2,
\end{align*}
where
\begin{align*} 
R_{N,d,\alpha, \bsgamma}(\bsz) = \left( \sum_{ \bsh\in \ZZ^d}\frac{1}{r^2_{\alpha,\bsgamma}(\bsh)}\sum_{\bsl \in P^{\perp}_{N,\bsz} \setminus \{\bszero\}} \frac{1}{r^2_{\alpha,\bsgamma}(\bsh+\bsl)} \right)^{1 / 2} .
\end{align*}
\end{lemma}

This lemma implies that generating vectors $\bsz$ with a small value of $R_{N,d,\alpha, \bsgamma}(\bsz)$ can yield a small worst-case error. To use $R_{N,d,\alpha, \bsgamma}(\bsz)$ as a quality criterion for constructing good $\bsz$, it is important that $R_{N,d,\alpha, \bsgamma}(\bsz)$ has a computable formula, as shown below.
\begin{remark}\label{rem:computable}
$\left[R_{N,d,\alpha, \bsgamma}(\bsz)\right]^2$ admits the representation
\begin{align*}
 \left[R_{N,d,\alpha, \bsgamma}(\bsz)\right]^2=-\prod_{j=1}^{d} (1+2\zeta(4\alpha)\gamma_j^4)+\frac{1}{N} \sum_{k=0}^{N-1} \prod_{j=1}^{d} \left(1+\gamma_j ^2\sum_{h\in\mathbb{Z}\setminus \{0 \} } \frac{\exp(2\pi i k z_j h/N)}{|h|^{2\alpha}} \right)^2 ,
\end{align*}
see \cite[Section~3.4]{dick2007lattice} and \cite[Remark~13.4]{dick2022lattice}, where $\zeta(x):=\sum_{i=1}^{\infty}i^{-x}$ for $x>1$ denotes the Riemann zeta function.
When $\alpha$ is an integer, we further have
\begin{align*}
 \left[R_{N,d,\alpha, \bsgamma}(\bsz)\right]^2=-\prod_{j=1}^{d} (1+2\zeta(4\alpha)\gamma_j^4)+\frac{1}{N} \sum_{k=0}^{N-1} \prod_{j=1}^{d} \left(1+\gamma_j ^2 \frac{(-1)^{\alpha+1}(2\pi)^{2\alpha}}{(2\alpha)!} B_{2\alpha}\left(\left\{\frac{kz_j}{N}\right\}\right) \right)^2 ,
\end{align*}
where $B_{2\alpha}$ denotes the Bernoulli polynomial of degree $2\alpha$. This simplification can be easily checked by the fact that $B_{2\alpha}$ has the absolutely convergent Fourier series
\begin{align*}
B_{2\alpha}(x)=\frac{2(-1)^{\alpha+1}(2\alpha)!}{(2\pi)^{2\alpha}}\sum_{k=1}^{\infty}\frac{\cos(2\pi kx)}{k^{2\alpha}}=\frac{(-1)^{\alpha+1}(2\alpha)!}{(2\pi)^{2\alpha}}\sum_{k\in\mathbb{Z}\backslash \{0\}}\frac{\exp(2\pi ik x)}{|k|^{2\alpha}},
\end{align*}
for any $x\in[0,1]$.
\end{remark}

\section{Construction algorithm}\label{sec:construction}
As mentioned in the introduction, in our randomized lattice-based algorithm for multivariate $L_2$-approximation, we draw the number of points $N$ randomly and then select a generating vector using a randomized CBC algorithm.
For $M \geq 2$, define a set of prime numbers
\begin{equation}
	\Pcal_M:=\{\text{$N$ is prime} \mid \lceil M / 2\rceil<N \leq M \}. \label{def:P_M}
\end{equation}
The number of points $N$ will be drawn from the uniform distribution over this set $\Pcal_M$.
The cardinality of $\Pcal_M$ is known to be lower bounded as $\left|\mathcal{P}_{M}\right| \geq c M / \log M$ for some absolute constant $c>0$, see \cite[Corollaries 1--3]{rosser1962approximate}. Let $\tau \in(0,1)$ be given. 
Similar to \cite{dick2022component}, consider the following randomized CBC algorithm to construct a good generating vector for a randomized lattice-based approximation. 
\begin{algorithm}\label{alg:random_cbc}
    For given $M, d \in \NN, \alpha > 1/2, \bsgamma \in[0,1]^{\NN}$ and $\tau \in(0,1)$, do the following:
    \begin{enumerate}
        \item Randomly draw $N \in \NN$ from the uniform distribution over the set $\Pcal_M$.
        \item Set $z_{1}=1$.
        \item \textbf{For} $s$ from $2$ to $d$, do the following:
        \begin{enumerate}
            \item Compute
            \[
            R_{N,s,\alpha,\bsgamma}\left(\bsz_{s-1}, z_{s}\right)=\left(\sum_{\bsh\in \ZZ^s} \sum_{\bsl \in P_{N,( \bsz_{s-1},z_s)}^{\perp}\setminus \{\bszero\}} \frac{1}{r^2_{\alpha, \bsgamma}(\bsh) r^2_{\alpha, \bsgamma}(\bsh+\bsl)}\right)^{1/2},
            \]
            for all $z_{s} \in\{1, \ldots, N-1\}$, where we write $\bsz_{s-1}=\left(z_{1}, \ldots, z_{s-1}\right)$.
            \item Construct a $\lceil\tau (N-1)\rceil$-element set $Z_{s} \subset\{1, \ldots, N-1\}$ such that $R_{N,s,\alpha,\bsgamma}\left(\bsz_{s-1}, \zeta\right) \leq R_{N,s,\alpha,\bsgamma}\left(\bsz_{s-1}, \eta\right)$ for all $\zeta \in Z_{s}$ and $\eta \in\{1, \ldots, N-1\} \setminus Z_{s}$.
            \item Randomly draw $z_{s}$ from the uniform distribution over the set $Z_{s}$.
        \end{enumerate}
        \textbf{end for}
    \end{enumerate}
\end{algorithm}

In the step~3.(b), we need to arrange the integers $1, \ldots, N-1$ such that the corresponding value $R_{N,s,\alpha,\bsgamma}$ is listed in ascending order and then pick one of the first $\lceil\tau(N-1)\rceil$ integers. 
This arrangement may not be unique if some of the integers yield the same value of $R_{N,s,\alpha,\bsgamma}$. 
However, we can always make the ordering unique by further arranging those integers in ascending order. 
Moreover, as shown in Remark~\ref{rem:computable}, when $\alpha$ is an integer, $R_{N,s,\alpha,\bsgamma}$ has a computable formula as
\begin{align*}
 \left[R_{N,s,\alpha, \bsgamma}\left(\bsz_{s-1}, z_{s}\right)\right]^2 & = -\prod_{j=1}^{s} (1+2\zeta(4\alpha)\gamma_j^4)\\
 & \quad +\frac{1}{N} \sum_{k=0}^{N-1} \theta_{\bsz_{s-1}, \alpha,\bsgamma}(k) \left(1+\gamma_s ^2 \frac{(-1)^{\alpha+1}(2\pi)^{2\alpha}}{(2\alpha)!} B_{2\alpha}\left(\left\{\frac{kz_s}{N}\right\}\right) \right)^2 ,
\end{align*}
where we write
\[ \theta_{\bsz_{s-1}, \alpha,\bsgamma}(k) = \prod_{j=1}^{s-1} \left(1+\gamma_j ^2 \frac{(-1)^{\alpha+1}(2\pi)^{2\alpha}}{(2\alpha)!} B_{2\alpha}\left(\left\{\frac{kz_j}{N}\right\}\right) \right)^2 . \]
By keeping $\theta_{\bsz_{s-1}, \alpha,\bsgamma}(k)$ for all $k=0,1,\ldots,N-1$, computing $\left[R_{N,s,\alpha, \bsgamma}\left(\bsz_{s-1}, z_{s}\right)\right]^2$ for all $z_s\in \{1,\ldots,N-1\}$ can be done with $O(N\log N)$ arithmetic operations with the help of the fast Fourier transform, according to the work by Nuyens and Cools \cite{nuyens2006fast}. Thus, the total cost for a single run of Algorithm~\ref{alg:random_cbc} is of order $dM\log M$.

For a fixed $N \in \Pcal_M$, let $\Zcal_{N,d,\tau}$ denote the set of possible generating vectors produced by Algorithm~\ref{alg:random_cbc}. 

For any $\bsz$ drawn by Algorithm~\ref{alg:random_cbc}, the following bound holds:
\begin{theorem}\label{thm:worst-case_bound}
Let $M, d \in \mathbb{N}$, $\alpha > 1 / 2$, $\bsgamma \in [0,1]^{\mathbb{N}}$, and $\tau \in (0,1)$ with $M \geq 4$ be given. For any $N \in \Pcal_M$ and $\bsz \in \Zcal_{N, d, \tau}$ drawn by Algorithm~\ref{alg:random_cbc}, we have
\begin{align*}
R_{N,s,\alpha, \bsgamma}(\bsz_s) \leq \left(\frac{1}{(1-\tau)(N-1)} \prod_{j=1}^{s} \left( 1+2^{2\alpha+2} \gamma_j^{1/\lambda} \zeta(\alpha/\lambda) \right)^2 \right)^{\lambda},
\end{align*}
for any $1/2 \leq \lambda < \alpha$ and $1 \leq s \leq d$.
\end{theorem}

In the following proof, we often use the subadditivity, a version of Jensen's inequality,
\begin{align}\label{eq:jensen}
    \left(\sum_{i=1}^{\infty}a_i\right)^c \leq \sum_{i=1}^{\infty}a_i^c,
\end{align}
which holds for any summable sequence $(a_i)_{i \in \NN}$ with $a_i \geq 0$ for all $i$ and any $0 < c \leq 1$.

\begin{proof}[Proof of Theorem~\ref{thm:worst-case_bound}] We prove the result by induction on $s$. For $s=1$, since we have
\begin{align*}
    \left(R_{N,1,\alpha,\bsgamma}(1)\right)^2=\sum_{h\in \ZZ} \sum_{\substack{\ell \in \ZZ\setminus \{0\}\\ \ell\equiv 0\Mod{N}}} \frac{1}{r^2_{\alpha, \gamma_1}(h) r^2_{\alpha, \gamma_1}(h+\ell)},
\end{align*}
applying Jensen's inequality \eqref{eq:jensen}, for any $1/2\leq \lambda<\alpha$ leads to 
\begin{align*}
    \left(R_{N,1,\alpha,\bsgamma}\left(1\right)\right)^{1/\lambda} & \leq \sum_{h\in \ZZ} \sum_{\substack{\ell \in \ZZ\setminus \{0\}\\ \ell\equiv 0\Mod{N}}} \frac{1}{(r_{\alpha, \gamma_1}(h) r_{\alpha, \gamma_1}(h+\ell))^{1/\lambda}}\\
    & = \sum_{\substack{h \in \ZZ\\ h\equiv 0\Mod{N}}} \sum_{\substack{\ell \in \ZZ\setminus \{0\}\\ \ell\equiv 0\Mod{N}}} \frac{1}{(r_{\alpha, \gamma_1}(h) r_{\alpha, \gamma_1}(h+\ell))^{1/\lambda}}\\
    & \quad + \sum_{\substack{h \in \ZZ\\ h\not\equiv 0\Mod{N}}} \sum_{\substack{\ell \in \ZZ\setminus \{0\}\\ \ell\equiv 0\Mod{N}}} \frac{1}{(r_{\alpha, \gamma_1}(h) r_{\alpha, \gamma_1}(h+\ell))^{1/\lambda}}.
\end{align*}
For the first double sum, we have
\begin{align*}
    & \sum_{\substack{h \in \ZZ\\ h\equiv 0\Mod{N}}} \sum_{\substack{\ell \in \ZZ\setminus \{0\}\\ \ell\equiv 0\Mod{N}}} \frac{1}{(r_{\alpha, \gamma_1}(h) r_{\alpha, \gamma_1}(h+\ell))^{1/\lambda}}\\
    & = \sum_{h \in \ZZ} \sum_{\ell \in \ZZ\setminus \{0\}} \frac{1}{(r_{\alpha,\gamma_1}(Nh) r_{\alpha, \gamma_1}(Nh+N\ell))^{1/\lambda}}\\
    & = \sum_{h \in \ZZ} \frac{1}{(r_{\alpha, \gamma_1}(Nh))^{1/\lambda}}\sum_{\ell \in \ZZ} \frac{1}{(r_{\alpha, \gamma_1}(Nh+N\ell))^{1/\lambda}}- \sum_{h \in \ZZ} \frac{1}{(r_{\alpha, \gamma_1}(Nh))^{2/\lambda}}\\
    & = \left(1+2\sum_{h=1}^{\infty}\frac{\gamma_1^{1/\lambda}}{|Nh|^{\alpha/\lambda}}\right)^2-\left( 1+2\sum_{h=1}^{\infty}\frac{\gamma_1^{2/\lambda}}{|Nh|^{2\alpha/\lambda}}\right)\\
    & = \frac{4\gamma_1^{1/\lambda}\zeta(\alpha/\lambda)}{N^{\alpha/\lambda}}+\frac{4\gamma_1^{2/\lambda}(\zeta(\alpha/\lambda))^2}{N^{2\alpha/\lambda}}-\frac{2\gamma_1^{2/\lambda}\zeta(2\alpha/\lambda)}{N^{2\alpha/\lambda}} \leq \frac{4\gamma_1^{1/\lambda}\zeta(\alpha/\lambda)}{N}+\frac{4\gamma_1^{2/\lambda}(\zeta(\alpha/\lambda))^2}{N^2}.
\end{align*}
For the second double sum, recall that we assume $M\geq 4$, so that $N>\lceil M / 2\rceil\geq 2$. Since $N\in \Pcal_M$ is prime, this means that $N$ must be always odd. Then it holds that
\begin{align*}
    & \sum_{\substack{h \in \ZZ\\ h\not\equiv 0\Mod{N}}} \sum_{\substack{\ell \in \ZZ\setminus \{0\}\\ \ell\equiv 0\Mod{N}}} \frac{1}{(r_{\alpha, \gamma_1}(h) r_{\alpha, \gamma_1}(h+\ell))^{1/\lambda}}\\
    & = \sum_{\substack{h \in \ZZ\\ h\not\equiv 0\Mod{N}}} \sum_{\ell \in \ZZ\setminus \{0\}} \frac{1}{(r_{\alpha, \gamma_1}(h) r_{\alpha, \gamma_1}(h+N\ell))^{1/\lambda}}\\
    & = \sum_{\substack{h \in \ZZ\\ h\not\equiv 0\Mod{N}}} \sum_{\ell \in \ZZ} \frac{1}{(r_{\alpha, \gamma_1}(h) r_{\alpha, \gamma_1}(h+N\ell))^{1/\lambda}}-\sum_{\substack{h \in \ZZ\\ h\not\equiv 0\Mod{N}}}\frac{1}{(r_{\alpha, \gamma_1}(h))^{2/\lambda}}\\
    & = \sum_{\substack{j=-(N-1)/2\\ j\neq 0}}^{(N-1)/2}\sum_{k\in \ZZ} \sum_{\ell \in \ZZ} \frac{1}{(r_{\alpha, \gamma_1}(Nk+j) r_{\alpha, \gamma_1}(Nk+j+N\ell))^{1/\lambda}}-\sum_{\substack{h \in \ZZ\\ h\not\equiv 0\Mod{N}}}\frac{1}{(r_{\alpha, \gamma_1}(h))^{2/\lambda}}\\
    & = \sum_{\substack{j=-(N-1)/2\\ j\neq 0}}^{(N-1)/2}\left(\sum_{k\in \ZZ} \frac{1}{(r_{\alpha, \gamma_1}(Nk+j))^{1/\lambda}}\right)^2-\sum_{\substack{j=-(N-1)/2\\ j\neq 0}}^{(N-1)/2}\sum_{k \in \ZZ}\frac{1}{(r_{\alpha, \gamma_1}(Nk+j))^{2/\lambda}}\\
    & = \gamma_1^{2/\lambda}\sum_{\substack{j=-(N-1)/2\\ j\neq 0}}^{(N-1)/2}\left(\left(\sum_{k\in \ZZ} \frac{1}{|Nk+j|^{\alpha/\lambda}}\right)^2-\sum_{k \in \ZZ}\frac{1}{|Nk+j|^{2\alpha/\lambda}}\right)\\
    & \leq \gamma_1^{2/\lambda}\sum_{\substack{j=-(N-1)/2\\ j\neq 0}}^{(N-1)/2}\left(\left(\frac{1}{|j|^{\alpha/\lambda}}+\sum_{k\in \ZZ\setminus \{0\}} \frac{1}{|Nk|^{\alpha/\lambda}|1+j/(Nk)|^{\alpha/\lambda}}\right)^2-\frac{1}{|j|^{2\alpha/\lambda}}\right)\\
    & \leq \gamma_1^{2/\lambda}\sum_{\substack{j=-(N-1)/2\\ j\neq 0}}^{(N-1)/2}\left(\left(\frac{1}{|j|^{\alpha/\lambda}}+\sum_{k\in \ZZ\setminus \{0\}} \frac{2^{\alpha/\lambda}}{|Nk|^{\alpha/\lambda}}\right)^2-\frac{1}{|j|^{2\alpha/\lambda}}\right)\\
    & = \gamma_1^{2/\lambda}\sum_{\substack{j=-(N-1)/2\\ j\neq 0}}^{(N-1)/2}\left(\left(\frac{1}{|j|^{\alpha/\lambda}}+\frac{2^{\alpha/\lambda+1}\zeta(\alpha/\lambda)}{N^{\alpha/\lambda}}\right)^2-\frac{1}{|j|^{2\alpha/\lambda}}\right)\\
    & = \gamma_1^{2/\lambda}\sum_{\substack{j=-(N-1)/2\\ j\neq 0}}^{(N-1)/2}\left(\frac{2^{\alpha/\lambda+2}\zeta(\alpha/\lambda)}{|j|^{\alpha/\lambda}N^{\alpha/\lambda}}+\frac{2^{2\alpha/\lambda+2}(\zeta(\alpha/\lambda))^2}{N^{2\alpha/\lambda}}\right)\\
    & \leq \gamma_1^{2/\lambda}\left(\frac{2^{\alpha/\lambda+2}(\zeta(\alpha/\lambda))^2}{N^{\alpha/\lambda}}+\frac{2^{2\alpha/\lambda+2}(\zeta(\alpha/\lambda))^2}{N^{2\alpha/\lambda-1}}\right)\leq \frac{2^{2\alpha/\lambda+3}\gamma_1^{2/\lambda}(\zeta(\alpha/\lambda))^2}{N}.
\end{align*}
Therefore, we get
\begin{align*}
    \left(R_{N,1,\alpha,\bsgamma}(1)\right)^{1/\lambda} & \leq \frac{4\gamma_1^{1/\lambda}\zeta(\alpha/\lambda)}{N}+\frac{4\gamma_1^{2/\lambda}(\zeta(\alpha/\lambda))^2}{N^2}+\frac{2^{2\alpha/\lambda+3}\gamma_1^{2/\lambda}(\zeta(\alpha/\lambda))^2}{N}\\
    & \leq \frac{1}{N}\left( 4\gamma_1^{1/\lambda}\zeta(\alpha/\lambda)+(2+2^{4\alpha+3})\gamma_1^{2/\lambda}(\zeta(\alpha/\lambda))^2\right)\\
    & \leq \frac{1}{N}\left[-1+\left( 1+\sqrt{2+2^{4\alpha+3}}\gamma_1^{1/\lambda}\zeta(\alpha/\lambda)\right)^2\right] \\
    & \leq \frac{1}{(1-\tau)(N-1)}\left( 1+2^{2\alpha+2}\gamma_1^{1/\lambda}\zeta(\alpha/\lambda)\right)^2,
\end{align*}
which proves the case $s=1$.

For the induction step, let $\bsz_{s-1}$ be the $(s-1)$-dimensional vector drawn by the first $s-1$ steps of the algorithm. For any $z_s\in \{1,\ldots,N-1\}$, by separating the cases with $\ell_s=0$ and $\ell_s\neq 0$, we have
\begin{align*}
\left(R_{N,s,\alpha, \bsgamma}\left(\bsz_{s-1}, z_{s}\right)\right)^{2} & =\sum_{\bsh\in \ZZ^s}\sum_{\bsl \in P_{N,( \bsz_{s-1},z_s)}^{\perp}\setminus \{\bszero\}} \frac{1}{r^2_{\alpha, \bsgamma}(\bsh) r^2_{\alpha, \bsgamma}(\bsh+\bsl)}\\
& = \sum_{\bsh\in \ZZ^s} \sum_{\substack{\bsl \in P_{N,( \bsz_{s-1},z_s)}^{\perp}\setminus \{\bszero\}\\ \ell_s=0}} \frac{1}{r^2_{\alpha, \bsgamma}(\bsh) r^2_{\alpha, \bsgamma}(\bsh+\bsl)}\\
& \quad + \sum_{\bsh\in \ZZ^s} \sum_{\substack{\bsl \in P_{N,( \bsz_{s-1},z_s)}^{\perp}\setminus \{\bszero\}\\ \ell_s\neq 0}} \frac{1}{r^2_{\alpha, \bsgamma}(\bsh) r^2_{\alpha, \bsgamma}(\bsh+\bsl)}\\
&=\sum_{h_s\in \ZZ} \frac{1}{r^4_{\alpha,\gamma_s}(h_s)} \times \sum_{\bsh \in \ZZ^{s-1}} \sum_{\bsl \in P_{N, \bsz_{s-1}}^{\perp}\setminus \{\bszero\}} \frac{1}{r^2_{\alpha, \bsgamma}(\bsh)r^2_{\alpha, \bsgamma}(\bsh+\bsl)} \\
&\quad +\sum_{h_s\in \ZZ} \sum_{\ell_s \in \ZZ\setminus \{0\}} \frac{1}{r^2_{\alpha, \bsgamma}(h_s)r^2_{\alpha, \bsgamma}(h_s+\ell_s)} \\
& \quad \quad \quad \times \sum_{\bsh \in \ZZ^{s-1}} \sum_{ \substack{\bsl\in\ZZ^{s-1} \\ \bsl\cdot \bsz_{s-1}\equiv-\ell_s z_s \Mod{N}}} \frac{1}{r^2_{\alpha, \bsgamma}(\bsh) r^2_{\alpha, \bsgamma}(\bsh+\bsl)}  \\
&= \left(1+2\gamma_s^4\zeta(4\alpha)\right)\left(R_{N,s-1,\alpha,\bsgamma}(\bsz_{s-1})\right)^{2}+B_{N,s,\alpha,\bsgamma}(\bsz_{s-1},z_s),
\end{align*}
where we write
\begin{align*}
B_{N,s,\alpha,\bsgamma}(\bsz_{s-1},z_s) & = \sum_{h_s\in \ZZ} \sum_{\ell_s \in \ZZ\setminus \{0\}} \frac{1}{r^2_{\alpha, \bsgamma}(h_s)r^2_{\alpha, \bsgamma}(h_s+\ell_s)}\\
& \quad \quad \quad \times \sum_{\bsh \in \ZZ^{s-1}} \sum_{ \substack{\bsl\in\ZZ^{s-1} \\ \bsl\cdot \bsz_{s-1}\equiv-\ell_s z_s \Mod{N}}} \frac{1}{r^2_{\alpha, \bsgamma}(\bsh) r^2_{\alpha, \bsgamma}(\bsh+\bsl)}.
\end{align*}

Using Jensen's inequality \eqref{eq:jensen} and separating the cases $\ell_s\equiv 0\Mod{N}$ and $\ell_s\not\equiv 0\Mod{N}$ further, the average of $(B_{N,s,\alpha,\bsgamma}(\bsz_{s-1},z_s))^{1/(2\lambda)}$ for any $1/2\leq \lambda<\alpha$ over all $z_s \in \{ 1, \ldots,N-1\} $ is bounded as
\begin{align}\label{eq:average_of_B}
& \frac{1}{N-1}\sum_{z_s=1}^{N-1} (B_{N,s,\alpha,\bsgamma}(\bsz_{s-1},z_s))^{1/(2\lambda)} \notag \\
&\leq \frac{1}{N-1}\sum_{z_s=1}^{N-1}\sum_{h_s\in \ZZ} \sum_{\ell_s \in \ZZ\setminus \{0\}} \frac{1}{(r_{\alpha, \bsgamma}(h_s)r_{\alpha, \bsgamma}(h_s+\ell_s))^{1/\lambda}} \notag \\
& \quad \quad \quad \times \sum_{\bsh \in \ZZ^{s-1}} \sum_{ \substack{\bsl\in\ZZ^{s-1} \\ \bsl\cdot \bsz_{s-1}\equiv-\ell_s z_s \Mod{N}}} \frac{1}{(r_{\alpha, \bsgamma}(\bsh) r_{\alpha, \bsgamma}(\bsh+\bsl))^{1/\lambda}} \notag \\
& = \sum_{h_s\in \ZZ} \sum_{\substack{\ell_s \in \ZZ\setminus \{0\}\\ \ell_s\equiv 0 \Mod{N}}} \frac{1}{(r_{\alpha, \bsgamma}(h_s)r_{\alpha, \bsgamma}(h_s+\ell_s))^{1/\lambda}} \notag \\
& \quad \quad \quad \times \sum_{\bsh \in \ZZ^{s-1}} \sum_{ \substack{\bsl\in\ZZ^{s-1} \\ \bsl\cdot \bsz_{s-1}\equiv 0 \Mod{N}}} \frac{1}{(r_{\alpha, \bsgamma}(\bsh) r_{\alpha, \bsgamma}(\bsh+\bsl))^{1/\lambda}} \notag \\
& \quad + \frac{1}{N-1}\sum_{h_s\in \ZZ} \sum_{\substack{\ell_s \in \ZZ\setminus \{0\}\\ \ell_s\not\equiv 0 \Mod{N}}} \frac{1}{(r_{\alpha, \bsgamma}(h_s)r_{\alpha, \bsgamma}(h_s+\ell_s))^{1/\lambda}} \notag \\
& \quad \quad \quad \times \sum_{\bsh \in \ZZ^{s-1}}\sum_{ \substack{\bsl\in\ZZ^{s-1} \\ \bsl\cdot \bsz_{s-1}\not\equiv 0 \Mod{N}}} \frac{1}{(r_{\alpha, \bsgamma}(\bsh) r_{\alpha, \bsgamma}(\bsh+\bsl))^{1/\lambda}} \notag \\
& = \sum_{h_s\in \ZZ} \sum_{\substack{\ell_s \in \ZZ\setminus \{0\}\\ \ell_s\equiv 0 \Mod{N}}} \frac{1}{(r_{\alpha, \bsgamma}(h_s)r_{\alpha, \bsgamma}(h_s+\ell_s))^{1/\lambda}}\times T_{\bsz_{s-1},\lambda} \notag \\
& \quad + \frac{1}{N-1}\sum_{h_s\in \ZZ} \sum_{\substack{\ell_s \in \ZZ\setminus \{0\}\\ \ell_s\not\equiv 0 \Mod{N}}} \frac{1}{(r_{\alpha, \bsgamma}(h_s)r_{\alpha, \bsgamma}(h_s+\ell_s))^{1/\lambda}}\times (\tilde{T}_{\lambda}-T_{\bsz_{s-1},\lambda}) \notag \\
& = \sum_{h_s\in \ZZ} \sum_{\substack{\ell_s \in \ZZ\setminus \{0\}\\ \ell_s\equiv 0 \Mod{N}}} \frac{1}{(r_{\alpha, \bsgamma}(h_s)r_{\alpha, \bsgamma}(h_s+\ell_s))^{1/\lambda}}\times \frac{N T_{\bsz_{s-1},\lambda}-\tilde{T}_{\lambda}}{N-1} \notag \\
& \quad + \frac{1}{N-1}\sum_{h_s\in \ZZ} \sum_{\ell_s \in \ZZ\setminus \{0\}} \frac{1}{(r_{\alpha, \bsgamma}(h_s)r_{\alpha, \bsgamma}(h_s+\ell_s))^{1/\lambda}}\times (\tilde{T}_{\lambda}-T_{\bsz_{s-1},\lambda}),
\end{align} 
where we write
\begin{align*}
    T_{\bsz_{s-1},\lambda} & := \sum_{\bsh \in \ZZ^{s-1}} \sum_{ \substack{\bsl\in\ZZ^{s-1} \\ \bsl\cdot \bsz_{s-1}\equiv 0 \Mod{N}}} \frac{1}{(r_{\alpha, \bsgamma}(\bsh) r_{\alpha, \bsgamma}(\bsh+\bsl))^{1/\lambda}},\\
    \tilde{T}_{\lambda} & := \sum_{\bsh \in \ZZ^{s-1}}\sum_{ \bsl\in\ZZ^{s-1}} \frac{1}{(r_{\alpha, \bsgamma}(\bsh) r_{\alpha, \bsgamma}(\bsh+\bsl))^{1/\lambda}}=\prod_{j=1}^{s-1}\left( 1+2\gamma_j^{1/\lambda}\zeta(\alpha/\lambda)\right)^2.
\end{align*}

For the first term of \eqref{eq:average_of_B}, it can be inferred from the result for the case $s=1$ that
\begin{align*}
    \sum_{h_s\in \ZZ} \sum_{\substack{\ell_s \in \ZZ\setminus \{0\}\\ \ell_s\equiv 0 \Mod{N}}} \frac{1}{(r_{\alpha, \bsgamma}(h_s)r_{\alpha, \bsgamma}(h_s+\ell_s))^{1/\lambda}} \leq \frac{1}{N}\left( 4\gamma_s^{1/\lambda}\zeta(\alpha/\lambda)+(2+2^{4\alpha+3})\gamma_s^{2/\lambda}(\zeta(\alpha/\lambda))^2\right).
\end{align*}
For the second term of \eqref{eq:average_of_B}, we have
\begin{align*}
    & \sum_{h_s\in \ZZ} \sum_{\ell_s \in \ZZ\setminus \{0\}} \frac{1}{(r_{\alpha, \bsgamma}(h_s)r_{\alpha, \bsgamma}(h_s+\ell_s))^{1/\lambda}}\\
    & = \sum_{h_s\in \ZZ} \sum_{\ell_s \in \ZZ} \frac{1}{(r_{\alpha, \bsgamma}(h_s)r_{\alpha, \bsgamma}(h_s+\ell_s))^{1/\lambda}}-\sum_{h_s\in \ZZ}\frac{1}{(r_{\alpha, \bsgamma}(h_s))^{2/\lambda}}\\
    & = \left( 1+2\gamma_s^{1/\lambda}\zeta(\alpha/\lambda)\right)^2-\left( 1+2\gamma_s^{2/\lambda}\zeta(2\alpha/\lambda)\right)\\
    & \leq 4\gamma_s^{1/\lambda}\zeta(\alpha/\lambda)+4\gamma_s^{2/\lambda}(\zeta(\alpha/\lambda))^2.
\end{align*}
These bounds tell us
\begin{align*}
    & \frac{1}{N-1}\sum_{z_s=1}^{N-1} (B_{N,s,\alpha,\bsgamma}(\bsz_{s-1},z_s))^{1/(2\lambda)} \\
    & \leq \frac{1}{N}\left( 4\gamma_s^{1/\lambda}\zeta(\alpha/\lambda)+(2+2^{4\alpha+3})\gamma_s^{2/\lambda}(\zeta(\alpha/\lambda))^2\right) \frac{N T_{\bsz_{s-1},\lambda}-\tilde{T}_{\lambda}}{N-1} \\
    & \quad + \frac{1}{N-1}\left( 4\gamma_s^{1/\lambda}\zeta(\alpha/\lambda)+4\gamma_s^{2/\lambda}(\zeta(\alpha/\lambda))^2\right) (\tilde{T}_{\lambda}-T_{\bsz_{s-1},\lambda})\\
    & \leq \frac{\tilde{T}_{\lambda}}{N-1}\left(8\gamma_s^{1/\lambda}\zeta(\alpha/\lambda)+(6+2^{4\alpha+3})\gamma_s^{2/\lambda}(\zeta(\alpha/\lambda))^2\right).
\end{align*}

Let us temporarily regard $z_s$ as a random variable following the uniform distribution over the set $\{1,\ldots,N-1\}$. 
The above argument shows an upper bound on the expected value of $(B_{N,s,\alpha,\bsgamma}(\bsz_{s-1},z_s))^{1/(2\lambda)}$ taken with respect to $z_s$. 
By applying Markov's inequality, we conclude that the probability of the inequality
\begin{align}\label{eq:bound_markov}
    (B_{N,s,\alpha,\bsgamma}(\bsz_{s-1},z_s))^{1/(2\lambda)}\leq \frac{\tilde{T}_{\lambda}}{(1-\tau)(N-1)}\left(8\gamma_s^{1/\lambda}\zeta(\alpha/\lambda)+(6+2^{4\alpha+3})\gamma_s^{2/\lambda}(\zeta(\alpha/\lambda))^2\right)
\end{align}
being satisfied is at least $\tau$.
This, in turn, implies that there exist at least $\left\lceil \tau(N-1) \right\rceil$ components $z_s \in \{1,\ldots,N-1 \}$ such that the bound \eqref{eq:bound_markov} holds.

Thus, using the induction hypothesis for $\bsz_{s-1}$, for any $N \in \Pcal_M$ and $\bsz \in \Zcal_{N,d,\tau}$ drawn by Algorithm~\ref{alg:random_cbc}, we have
\begin{align*}
    & \left(R_{N,s,\alpha, \bsgamma}\left(\bsz_{s-1}, z_{s}\right)\right)^{1/\lambda}\\
    & \leq \left(1+2\gamma_s^4\zeta(4\alpha)\right)^{1/(2\lambda)}(R_{N,s-1,\alpha,\bsgamma}(\bsz_{s-1}))^{1/\lambda}+(B_{N,s,\alpha,\bsgamma}(\bsz_{s-1},z_s))^{1/(2\lambda)}\\
    & \leq \left(1+2^{1/(2\lambda)}\gamma_s^{2/\lambda}(\zeta(4\alpha))^{1/(2\lambda)}\right)(R_{N,s-1,\alpha,\bsgamma}(\bsz_{s-1}))^{1/\lambda}\\
    & \quad +\frac{\tilde{T}_{\lambda}}{(1-\tau)(N-1)}\left(8\gamma_s^{1/\lambda}\zeta(\alpha/\lambda)+(6+2^{4\alpha+3})\gamma_s^{2/\lambda}(\zeta(\alpha/\lambda))^2\right)\\
    & \leq \frac{1}{(1-\tau)(N-1)}\left(1+8\gamma_s^{1/\lambda}\zeta(\alpha/\lambda)+(8+2^{4\alpha+3})\gamma_s^{2/\lambda}(\zeta(\alpha/\lambda))^2\right) \prod_{j=1}^{s-1}\left( 1+2^{2\alpha+2}\gamma_j^{1/\lambda}\zeta(\alpha/\lambda)\right)^2\\
    & \leq \frac{1}{(1-\tau)(N-1)}\prod_{j=1}^{s}\left( 1+2^{2\alpha+2}\gamma_j^{1/\lambda}\zeta(\alpha/\lambda)\right)^2.
\end{align*}
This proves the result.
\end{proof}

\section{Randomized lattice-based approximation}\label{sec:upper}

\subsection{Our algorithm}
In our randomized lattice-based approximation, we draw the number of points $N$ and the generating vector $\bsz$ randomly according to Algorithm~\ref{alg:random_cbc}, and also apply a random shift to estimate the Fourier coefficients $\hat{f}(\bsh)$ with $\bsh\in \Acal_d(T)$:
\[ A^{\rand}_{N,\bsz,\bsDelta,\Acal_d(T)}(f)(\bsx)= \sum_{\bsh \in \Acal_d(T)} \hat{f}_{N,\bsz,\bsDelta}(\bsh) \exp (2 \pi i \bsh \cdot \bsx),\]
where $N,\bsz,\bsDelta$ are the random variables and $\bsDelta$ follows the uniform distribution over $[0,1)^d$, and 
\[ \hat{f}_{N,\bsz,\bsDelta}(\bsh)=\frac{1}{N} \sum_{k=0}^{N-1} f\left( \left\{ \frac{k \bsz}{N}+\bsDelta \right\}\right)\exp(-2\pi i \bsh \cdot (k \bsz/N+\bsDelta)).\]
Instead of the deterministic worst-case $L_2$-approximation error, we employ the worst-case root mean squared $L_2$-approximation error:
\begin{align*}
    & e^{\rms\text{-}L_2\text{-}\app}_{d, \alpha, \bsgamma}(A^{\rand}_{N,\bsz,\bsDelta,\Acal_d(T)}) \\
    &  :=\sup_{\substack{ f\in H_{d, \alpha, \bsgamma} \\ \|f\|_{d, \alpha, \bsgamma} \leq 1}} \sqrt{\EE\left[\| f-A^{\rand}_{N,\bsz,\bsDelta,\Acal_d(T)}(f)\|^2_{L_2}\right] }. \\
    & \: =\sup_{ \substack{f \in H_{d,\alpha,\bsgamma} \\ \|f\|_{d,\alpha,\bsgamma} \leq 1}}\sqrt{\frac{1}{|\Pcal_M|}\sum_{N \in \Pcal_M} \frac{1}{| \Zcal_{N,d,\tau}|} \sum_{\bsz \in \Zcal_{N,d,\tau}} \int_{[0,1)^d}\| f-A^{\rand}_{N,\bsz,\bsDelta,\Acal_d(T)}(f)\|_{L_2} ^2\, \mathrm{d} \bsDelta}
\end{align*}
as an error criterion. 
In passing, we note that our randomized algorithm is biased; that is, we generally have $\EE[A^{\rand}_{N,\bsz,\bsDelta,\Acal_d(T)}(f)] \neq f$, due to the truncation of the set of Fourier coefficients.

\begin{remark}
    Although our randomized method is based on Algorithm~\ref{alg:random_cbc}, it can be substituted with Algorithm~2.4 in \cite{dick2022component}.
    The difference lies in the error criterion being minimized.
    Although we omit a detailed discussion in this paper, the randomized method based on Algorithm~2.4 in \cite{dick2022component} leads to an upper bound of the same order as the one we prove in Theorem~4.5 and Corollary~4.6; namely, the $M^{-\alpha(2\alpha+1)/(4\alpha+1)+\varepsilon}$ decay of the worst-case RMSE. The major advantage of using Algorithm~\ref{alg:random_cbc} lies in its superior performance in the deterministic setting: for each realization of the number of points $N$ and the generating vector $\bsz$, the corresponding deterministic, single rank-1 lattice-based approximation algorithm yields a better convergence rate for the worst-case $L_2$-approximation error. The difference in convergence rates between the two deterministic lattice-based methods is discussed in detail in \cite[Chapter~13]{dick2022lattice}, particularly in the paragraphs titled ``A direct approach'' and ``Improving the convergence rate.'' Consequently, our randomized method based on Algorithm~\ref{alg:random_cbc} is expected to exhibit more stable and reliable approximation behavior.
\end{remark}

\subsection{Error Analysis}
In what follows, for $\bsl \in \ZZ^d$, we write
\begin{align*}
\omega(\bsl) := \frac{1}{|\Pcal_M|} \sum_{N \in \Pcal_M} \frac{1}{|\Zcal_{N,d,\tau}|} \sum_{\bsz \in \Zcal_{N,d,\tau}} \II_{\bsl \in P_{N,\bsz}^{\perp}},
\end{align*}
where $\II_{\bsl \in P_{N,\bsz}^{\perp}}$ denotes the indicator function that returns $1$ if $\bsl \in P_{N,\bsz}^{\perp}$ holds and $0$ otherwise. 
Note that $\omega(\bsl)$ is nothing but the probability that a given frequency $\bsl$ belongs to the dual lattice $P_{N,\bsz}^{\perp}$ when $\bsz$ follows the uniform distribution over $\Zcal_{N,d,\tau}$ given $N$, and $N$ follows the uniform distribution over $\Pcal_M$. Note that the size of $\Zcal_{N,d,\tau}$ can be bounded below as
\[ |\Zcal_{N,d,\tau}| = 1\times \left( \lceil \tau (N-1) \rceil \right)^{d-1} \geq \tau^{d-1} (N-1)^{d-1}. \]

We show some lemmas first.
\begin{lemma}\label{lem:random_error_bound}
    It holds that
    \[ \left(e^{\rms\text{-}L_2\text{-}\app}_{d, \alpha, \bsgamma}(A^{\rand}_{N,\bsz,\bsDelta,\Acal_d(T)})\right)^2\leq \sup_{\bsh\notin \mathcal{A}_d(T)}\frac{1}{r^2_{\alpha, \bsgamma}(\bsh)}+\sum_{\bsh\in \Acal_d(T)}\sup_{\bsl \in \ZZ^d\setminus \{\bszero\}}\frac{\omega(\bsl)}{r^2_{\alpha, \bsgamma}(\bsh+\bsl)}.\]
\end{lemma}

\begin{proof}
    Let us consider the Fourier series of $f\in H_{d,\alpha,\bsgamma}$. For fixed $N$ and $\bsz$, the orthonormality of the Fourier system leads to
    \begin{align*}
        & \int_{[0,1)^d}\| f-A^{\rand}_{N,\bsz,\bsDelta,\Acal_d(T)}(f)\|_{L_2} ^2\, \mathrm{d} \bsDelta \\
        & = \int_{[0,1)^d}\int_{[0,1)^d} \left(f(\bsx)-A^{\rand}_{N,\bsz,\bsDelta,\Acal_d(T)}(f)(\bsx)\right)^2\, \mathrm{d} \bsx\, \mathrm{d} \bsDelta\\
        &=\int_{[0,1)^d}\int_{[0,1)^d} \left(\sum_{\bsh\notin \mathcal{A}_d(T)} \hat{f}(\bsh)\exp(2\pi i \bsh \cdot \bsx)\right.\\
        & \qquad \qquad \qquad \left. + \sum_{\bsh\in \Acal_d(T)}\left( \hat{f}(\bsh)-\hat{f}_{N,\bsz,\bsDelta}(\boldsymbol{h})\right)\exp(2\pi i \boldsymbol{h}\cdot \bsx)\right)^2\, \mathrm{d} \bsx\, \mathrm{d} \bsDelta\\
        & = \int_{[0,1)^d}\left(\sum_{\bsh\notin \mathcal{A}_d(T)} |\hat{f}(\bsh)|^2+\sum_{\bsh\in \Acal_d(T)}|\hat{f}(\bsh)-\hat{f}_{N,\bsz,\bsDelta}(\bsh)|^2\right)\, \mathrm{d} \bsDelta\\
        & = \sum_{\bsh\notin \mathcal{A}_d(T)} |\hat{f}(\bsh)|^2+\sum_{\bsh\in \Acal_d(T)}\int_{[0,1)^d}|\hat{f}(\bsh)-\hat{f}_{N,\bsz,\bsDelta}(\bsh)|^2\,  \mathrm{d} \bsDelta.
    \end{align*}
     For the difference between the true Fourier coefficient and its estimation by the shifted lattice rule $ \hat{f}(\boldsymbol{h})-\hat{f}_{N,\boldsymbol{z},\bsDelta}(\boldsymbol{h})$ for each $\bsh\in \Acal_d(T)$, using the pointwise representation of $f$ by the Fourier series again and Lemma~\ref{lem:character}, we have
    \begin{align*}
        \hat{f}(\boldsymbol{h})-\hat{f}_{N,\bsz,\bsDelta}(\bsh) &=\hat{f}(\boldsymbol{h})-\frac{1}{N}\sum_{n=0}^{N-1} \left(\sum_{\boldsymbol{\ell}\in \mathbb{Z}^d} \hat{f}(\bsl)\exp(2\pi i \bsl\cdot (\bsx_n+\bsDelta)) \right)\exp(-2\pi i \bsh\cdot(\bsx_n+\bsDelta))\\
        &=\hat{f}(\bsh)-\sum_{\bsl\in \ZZ^d} \hat{f}(\bsl) \exp(2\pi i (\bsl-\bsh)\cdot\bsDelta)) \left(\frac{1}{N}\sum_{n=0}^{N-1} \exp(2\pi i (\bsl-\bsh )\cdot \bsx_n) \right) \\
        &=\hat{f}(\bsh)-\sum_{\substack{\bsl\in \ZZ^d \\ \bsl-\bsh\in P_{N,\bsz}^\perp}} \hat{f}(\bsl)\exp( 2\pi i(\bsl-\bsh)\cdot\bsDelta)\\
        &=\sum_{\substack{\bsl\in \ZZ^d \setminus \{\bsh \}\\ \bsl-\bsh\in P_{N,\bsz}^\perp}} \hat{f}(\bsl)\exp( 2\pi i(\bsl-\bsh)\cdot\bsDelta) \\
        &=\sum_{\bsl \in P_{N,\boldsymbol{z}}^\perp\setminus \{\bszero\}} \hat{f}(\bsl+\bsh)\exp( 2\pi i \bsl\cdot\bsDelta).
    \end{align*}
    Due to the orthonormality of the Fourier system, we have
    \begin{align*}
        & \int_{[0,1)^d}\| f-A^{\rand}_{N,\bsz,\bsDelta,\Acal_d(T)}(f)\|_{L_2} ^2\, \mathrm{d} \bsDelta \\
        & = \sum_{\bsh\notin \mathcal{A}_d(T)} |\hat{f}(\bsh)|^2+\sum_{\bsh\in \Acal_d(T)}\int_{[0,1)^d}\left|\sum_{\bsl \in P_{N,\boldsymbol{z}}^\perp\setminus \{\bszero\}} \hat{f}(\bsl+\bsh)\exp( 2\pi i \bsl\cdot\bsDelta)\right|^2\,  \mathrm{d} \bsDelta\\
        & = \sum_{\bsh\notin \mathcal{A}_d(T)} |\hat{f}(\bsh)|^2+\sum_{\bsh\in \Acal_d(T)}\sum_{\bsl \in P_{N,\boldsymbol{z}}^\perp\setminus \{\bszero\}}|\hat{f}(\bsl+\bsh)|^2.
    \end{align*}

    Therefore, for any $f\in H_{d,\alpha,\bsgamma}$, it holds that
    \begin{align*}
        & \frac{1}{|\Pcal_M|}\sum_{N \in \Pcal_M} \frac{1}{| \Zcal_{N,d,\tau}|} \sum_{\bsz \in \Zcal_{N,d,\tau}} \int_{[0,1)^d}\| f-A^{\rand}_{N,\bsz,\bsDelta,\Acal_d(T)}(f)\|_{L_2} ^2\, \mathrm{d} \bsDelta \\
       & = \frac{1}{|\Pcal_M|}\sum_{N \in \Pcal_M} \frac{1}{| \Zcal_{N,d,\tau}|} \sum_{\bsz \in \Zcal_{N,d,\tau}} \left( \sum_{\bsh\notin \Acal_d(T)} |\hat{f}(\bsh)|^2+\sum_{\bsh\in \Acal_d(T)}\sum_{\bsl \in P_{N,\boldsymbol{z}}^\perp\setminus \{\bszero\}}|\hat{f}(\bsl+\bsh)|^2\right)\\
       & = \sum_{\bsh\notin \Acal_d(T)} |\hat{f}(\bsh)|^2+\sum_{\bsh\in \Acal_d(T)}\sum_{\bsl \in \ZZ^d\setminus \{\bszero\}}|\hat{f}(\bsl+\bsh)|^2 \times \frac{1}{|\Pcal_M|} \sum_{N \in \Pcal_M} \frac{1}{|\Zcal_{N,d,\tau}|} \sum_{\bsz \in \Zcal_{N,d,\tau}} \II_{\bsl \in P_{N,\bsz}^{\perp}}\\
       & = \sum_{\bsh\notin\Acal_d(T)} |\hat{f}(\bsh)|^2+\sum_{\bsh\in \Acal_d(T)}\sum_{\bsl \in \ZZ^d\setminus \{\bszero\}}|\hat{f}(\bsl+\bsh)|^2\omega(\bsl)\\
       & \leq \left(\sum_{\bsh\notin \mathcal{A}_d(T)} |\hat{f}(\bsh)|^2r^2_{\alpha, \bsgamma}(\bsh)\right) \sup_{\bsh\notin \Acal_d(T)}\frac{1}{r^2_{\alpha, \bsgamma}(\bsh)}\\
       & \quad +\sum_{\bsh\in \Acal_d(T)}\left(\sum_{\bsl \in \ZZ^d\setminus \{\bszero\}}|\hat{f}(\bsl+\bsh)|^2r^2_{\alpha, \bsgamma}(\bsl+\bsh)\right) \sup_{\bsl \in \ZZ^d\setminus \{\bszero\}}\frac{\omega(\bsl)}{r^2_{\alpha, \bsgamma}(\bsl+\bsh)}\\
       & \leq \|f\|_{d,\alpha,\bsgamma}^2\left( \sup_{\bsh\notin \mathcal{A}_d(T)}\frac{1}{r^2_{\alpha, \bsgamma}(\bsh)}+\sum_{\bsh\in \Acal_d(T)}\sup_{\bsl \in \ZZ^d\setminus \{\bszero\}}\frac{\omega(\bsl)}{r^2_{\alpha, \bsgamma}(\bsh+\bsl)}\right).
    \end{align*}
    Thus we are done.
\end{proof}

Let us define 
\begin{equation}
    H_M:=\inf_{1/2\leq \lambda< \alpha} \left(\frac{2}{(1-\tau)M}\prod_{j=1}^{d}\left( 1+2^{2\alpha+2}\gamma_j^{1/\lambda}\zeta(\alpha/\lambda)\right)^2\right)^{\lambda}.\label{def:H_M}
\end{equation}
It follows from Theorem~\ref{thm:worst-case_bound} that, for any $N$ and $\bsz$ generated by Algorithm~\ref{alg:random_cbc}, we have $R_{N,s,\alpha, \bsgamma}(\bsz) \leq H_M$.
For the rest of this paper, we assume that 
\begin{align}\label{eq:assumption_number_points}
     M\geq \inf_{1/2\leq \lambda< \alpha}\frac{2}{1-\tau}\prod_{j=1}^{d}\left( 1+2^{2\alpha+2}\gamma_j^{1/\lambda}\zeta(\alpha/\lambda)\right)^2
\end{align}
holds so that $H_M\leq 1$.

\begin{remark}
	Because the number of quadrature points $N$ needs to be at least as large as $M/2$ (see \eqref{def:P_M}), it is important to discuss how the lower bound on $M$, given in \eqref{eq:assumption_number_points}, depends on $d$. It is clear that if $\gamma_1=\gamma_2=\cdots=c$ for some $0<c\leq 1$, it grows exponentially fast with $d$. Therefore, the more interesting case arises when $\gamma_j$ exhibits some decay. Using the elementary inequality $1+x\leq \exp(x)$, we obtain
    \begin{align*}
        \prod_{j=1}^{d}\left( 1+2^{2\alpha+2}\gamma_j^{1/\lambda}\zeta(\alpha/\lambda)\right) & \leq \exp\left( 2^{2\alpha+2}\zeta(\alpha/\lambda)\sum_{j=1}^{d}\gamma_j^{1/\lambda}\right)\\
        & \leq \exp\left( 2^{2\alpha+2}\zeta(\alpha/\lambda)\sum_{j=1}^{d}\gamma_j^{1/\alpha}\right).
    \end{align*}
    This implies that if $\sum_{j=1}^{\infty}\gamma_j^{1/\alpha}<\infty$, the lower bound on $M$ is bounded above independently of $d$. Furthermore, if $\lim\sup_{d\to \infty}\frac{1}{\log d}\sum_{j=1}^{d}\gamma_j^{1/\alpha}<\infty$, the lower bound on $M$ grows only polynomially with $d$. 
\end{remark}

\begin{lemma}\label{lem:bound_on_omega}
   Let $M, d \in \NN, \alpha > 1/2, \bsgamma \in[0,1]^{\NN}$ and $\tau \in(0,1)$ with $M\geq 4$ be given. Assume that \eqref{eq:assumption_number_points} holds. Then the following holds:
   \begin{enumerate}
        \item There exists a constant $c>0$ such that, for all $\bsl \in \ZZ^d \setminus \{\bszero\}$ we have
        \[ \omega(\bsl) \leq c \frac{\log \left(1+\|\bsl\|_{\infty}\right)}{\tau M}, \]
        where $\|\bsl\|_{\infty}=\max_j|\ell_j|$ denotes the maximum norm.
        \item If $\bsl \in \ZZ^d \setminus \{\bszero\}$ satisfies        \begin{align}\label{eq:condition_on_omega_zero}
        \sum_{\bsh\in \ZZ^d} \frac{1}{r^2_{\alpha, \bsgamma}(\bsh) r^2_{\alpha, \bsgamma}(\bsh+\bsl)} > H_M^2,
        \end{align}
        then $\omega(\bsl)=0$. Here, $H_M$ is defined in \eqref{def:H_M}.
   \end{enumerate}
\end{lemma}

\begin{proof}
    Since the first assertion was proven in \cite[Theorem~3.1]{dick2022component}, we only prove the second assertion. It follows from Theorem~\ref{thm:worst-case_bound} that, for any $N \in \Pcal_M$ and $\bsz \in \Zcal_{N,d,\tau}$ drawn by Algorithm~\ref{alg:random_cbc}, we have 
    \begin{align*}
        (R_{N,s,\alpha, \bsgamma}(\bsz))^2 = \sum_{\bsl \in P^{\perp}_{N,\bsz} \setminus \{\bszero\}} \left(\sum_{ \bsh\in \ZZ^d}\frac{1}{r^2_{\alpha,\bsgamma}(\bsh)r^2_{\alpha,\bsgamma}(\bsh+\bsl)}\right)\leq H_M^2.
    \end{align*}
    This means that, for any $\bsl \in \ZZ^d \setminus \{\bszero\}$ with \eqref{eq:condition_on_omega_zero}, it holds that $\bsl \notin P_{N,\bsz}^{\perp}$ for all $N \in \Pcal_M$ and all $\bsz \in \Zcal_{N,d,\tau}.$ This proves $\omega(\bsl)=0$ when the condition \eqref{eq:condition_on_omega_zero} holds.
\end{proof}

As one of the main results of this paper, we prove the following theorem.

\begin{theorem}\label{thm:rms_error_bound}
Let $M, d \in \NN, \alpha > 1/2, \bsgamma \in[0,1]^{\NN}$ and $\tau \in(0,1)$ with $M\geq 4$ be given. Assume that \eqref{eq:assumption_number_points} holds. We have
\begin{align*}
& \left(e^{\rms\text{-}L_2\text{-}\app}_{d, \alpha, \bsgamma}(A^{\rand}_{N,\bsz,\bsDelta,\Acal_d(T)})\right)^2\\
& \leq \frac{1}{T}+ C_{\alpha,\beta,\lambda,\tau}\frac{T^{1+1/(2\lambda)}}{M^{2\lambda-\lambda \beta/\alpha+1}}\prod_{j=1}^{d}\left( 1+2^{2\alpha+2}\gamma_j^{1/\lambda}\zeta(\alpha/\lambda)\right)^{4\lambda-2\lambda\beta/\alpha}\max(1,2^{\beta} \gamma_j^{\beta/\alpha}),
\end{align*}
for any $\lambda\in [1/2,\alpha)$ and $\beta\in (0,1]$, with a constant $C_{\alpha,\beta,\lambda,\tau}>0$ independent of $M,d,\bsgamma,T$.
\end{theorem}
\begin{proof}
It follows from Lemma~\ref{lem:random_error_bound} and the definition of $\Acal_d(T)$ that
\begin{align*}
    \left(e^{\rms\text{-}L_2\text{-}\app}_{d, \alpha, \bsgamma}(A^{\rand}_{N,\bsz,\bsDelta,\Acal_d(T)})\right)^2 \leq \frac{1}{T}+\sum_{\bsh\in \Acal_d(T)}\sup_{\bsl \in \ZZ^d\setminus \{\bszero\}}\frac{\omega(\bsl)}{r^2_{\alpha, \bsgamma}(\bsh+\bsl)}.
\end{align*}
In what follows, we denote the sum over $\bsh$ on the right-most side above by $B_M^2$. 
Now, note that the second assertion of Lemma~\ref{lem:bound_on_omega} implies that $\omega(\bsl)=0$ for $\bsl\in \ZZ^d\setminus \{\bszero\}$ if there exists an $\bsh$ such that 
\[ \frac{1}{r^2_{\alpha, \bsgamma}(\bsh) r^2_{\alpha, \bsgamma}(\bsh+\bsl)} > H^2_M.\]
Using this and the first assertion of Lemma~\ref{lem:bound_on_omega}, together with the elementary inequality $\log(1+x)\leq x^{\beta}/\beta$ for $x>0$ and $0<\beta\le 1$ with $\log$ denoting the natural logarithm, we have
\begin{align*}
    B_M^2 & = \sum_{\bsh\in \Acal_d(T)}\sup_{\bsl \in \ZZ^d\setminus \{\bszero\}}\frac{\omega(\bsl)}{r^2_{\alpha, \bsgamma}(\bsh+\bsl)} \\
    & \leq \frac{c}{\tau M}\sum_{\bsh\in \Acal_d(T)}\sup_{\substack{\bsl \in \ZZ^d\setminus \{\bszero\}\\ r_{\alpha,\bsgamma}(\bsh)r_{\alpha,\bsgamma}(\bsh+\bsl)\geq H^{-1}_M }}\frac{\log \left(1+\|\bsl\|_{\infty}\right)}{r^2_{\alpha, \bsgamma}(\bsh+\bsl)} \\
    & \leq \frac{c}{\beta \tau M}\sum_{\bsh\in \Acal_d(T)}\sup_{\substack{\bsl \in \ZZ^d\setminus \{\bszero\}\\ r_{\alpha,\bsgamma}(\bsh)r_{\alpha,\bsgamma}(\bsh+\bsl)\geq H^{-1}_M}}\frac{\|\bsl\|_{\infty}^{\beta}}{r^2_{\alpha, \bsgamma}(\bsh+\bsl)} \\
    & \leq \frac{c}{\beta \tau M}\sum_{\bsh\in \Acal_d(T)}\sup_{\substack{\bsl \in \ZZ^d\setminus \{\bszero\}\\ r_{\alpha,\bsgamma}(\bsh)r_{\alpha,\bsgamma}(\bsh+\bsl)\geq H^{-1}_M }}\frac{\left(r_{\alpha,\bsgamma}^2(\bsl)\right)^{\beta/(2\alpha)}}{r^2_{\alpha, \bsgamma}(\bsh+\bsl)},
\end{align*}
where the last inequality follows from the assumption that $\gamma_j \leq 1$ for all $j$, which ensures that
\[ \|\bsl\|_{\infty}^{\beta}\leq \prod_{\substack{j=1\\ \ell_j\neq 0}}^{d}|\ell_j|^{\beta}\leq \left(\prod_{\substack{j=1\\ \ell_j\neq 0}}^{d}\frac{|\ell_j|^{\alpha}}{\gamma_j}\right)^{\beta/\alpha}=\left(r_{\alpha,\bsgamma}^2(\bsl)\right)^{\beta/(2\alpha)} \]
for any $\bsl\in \ZZ^d\setminus \{\bszero\}.$
As mentioned in \cite[Eq.~(13.7)]{dick2022lattice}, for any pair of indices $\bsh, \bsl\in \ZZ^d$
\begin{align*}
r_{\alpha,\bsgamma}^2(\bsl)\leq r_{\alpha,\bsgamma}^2(\bsh)r_{\alpha,\bsgamma}^2(\bsh+\bsl) \prod_{j=1}^d \max(1,2^{2\alpha} \gamma_j^2 )
\end{align*}
holds, 
so we further have
\begin{align*}
    B_M^2 & \leq \frac{c}{\beta \tau M}\prod_{j=1}^d \max(1,2^{\beta} \gamma_j^{\beta/\alpha} )\sum_{\bsh\in \Acal_d(T)}\sup_{\substack{\bsl \in \ZZ^d\setminus \{\bszero\}\\ r_{\alpha,\bsgamma}(\bsh)r_{\alpha,\bsgamma}(\bsh+\bsl)\geq H^{-1}_M }}\frac{\left(r_{\alpha,\bsgamma}(\bsh)\right)^{\beta/\alpha}}{(r_{\alpha, \bsgamma}(\bsh+\bsl))^{2-\beta/\alpha}}\\
    & \leq \frac{c}{\beta \tau M}\prod_{j=1}^d \max(1,2^{\beta} \gamma_j^{\beta/\alpha} )\sum_{\bsh\in \Acal_d(T)}\sup_{\bsl \in \ZZ^d\setminus \{\bszero\}}r_{\alpha,\bsgamma}^2(\bsh)\, H_M^{2-\beta/\alpha}\\
    & = \frac{c}{\beta \tau M}H_M^{2-\beta/\alpha}\prod_{j=1}^d \max(1,2^{\beta} \gamma_j^{\beta/\alpha} )\sum_{\bsh\in \Acal_d(T)}r_{\alpha,\bsgamma}^2(\bsh) ,
\end{align*}
where the second inequality follows from the condition $r_{\alpha,\bsgamma}(\bsh)r_{\alpha,\bsgamma}(\bsh+\bsl)\geq H^{-1}_M$, and the last equality holds because the expression inside the supremum over $\bsl$ does not depend on $\bsl$.

For the sum over $\bsh$ above, we use \cite[Lemma~1]{kuo2006lattice} to obtain
\begin{align*}
    \sum_{\bsh\in \Acal_d(T)}r_{\alpha,\bsgamma}^2(\bsh) 
    \leq T\, |\Acal_d(T)| \leq T^{1+1/(2\lambda)}\prod_{j=1}^{d}\left(1+2\gamma_j^{1/\lambda}\zeta(\alpha/\lambda) \right),
\end{align*}
which holds for any $1/2<\lambda<\alpha$.
Therefore, we get an upper bound on $B_M^2$ as
\begin{align*}
    B_M^2 & \leq T^{1+1/(2\lambda)}\frac{c}{\beta \tau M}H_M^{2-\beta/\alpha}\prod_{j=1}^{d} \max(1,2^{\beta} \gamma_j^{\beta/\alpha}) \left( 1+2\gamma_j^{1/\lambda}\zeta(\alpha/\lambda)\right)\\
    & \leq C_{\alpha,\beta,\lambda,\tau}\frac{T^{1+1/(2\lambda)}}{M^{2\lambda-\lambda \beta/\alpha+1}}\prod_{j=1}^{d}\left( 1+2^{2\alpha+2}\gamma_j^{1/\lambda}\zeta(\alpha/\lambda)\right)^{4\lambda-2\lambda\beta/\alpha}\max(1,2^{\beta} \gamma_j^{\beta/\alpha}) .
\end{align*}
Thus we are done.
\end{proof}

By balancing the two terms appearing in the error bound shown in Theorem~\ref{thm:rms_error_bound}, we obtain the following error estimate in $M$.
\begin{corollary}\label{cor:final_result}
    Let $M, d \in \NN, \alpha > 1/2, \bsgamma \in[0,1]^{\NN}$ and $\tau \in(0,1)$ with $M\geq 4$ be given. Assume that \eqref{eq:assumption_number_points} holds. For any fixed $\lambda\in [1/2,\alpha)$ and $\beta\in (0,1]$, by choosing $T=M^{(2\lambda-\lambda\beta/\alpha+1)/(2+1/(2\lambda))}$, we have
    \begin{align*}
    & e^{\rms\text{-}L_2\text{-}\app}_{d, \alpha, \bsgamma}(A^{\rand}_{N,\bsz,\bsDelta,\Acal_d(T)})\\
    & \leq \frac{1}{M^{\lambda(2\lambda-\lambda \beta/\alpha+1)/(4\lambda+1)}}\sqrt{1+C_{\alpha,\beta,\lambda,\tau}\prod_{j=1}^{d}\left( 1+2^{2\alpha+2}\gamma_j^{1/\lambda}\zeta(\alpha/\lambda)\right)^{4\lambda-2\lambda\beta/\alpha}\max(1,2^{\beta} \gamma_j^{\beta/\alpha})}, 
    \end{align*}
    with a constant $C_{\alpha,\beta,\lambda,\tau}>0$ independent of $M,d,\bsgamma$.
\end{corollary}

Let us consider the case where $\lambda\to \alpha-$ and $\beta\to 0+$. The rate of the worst-case RMSE we obtained is of order $M^{-\alpha(2\alpha+1)/(4\alpha+1)+\varepsilon}$ for an arbitrarily small $\varepsilon>0$, which is no worse than $M^{-\alpha/2-1/12+\varepsilon}$ for any $\alpha>1/2$.
Moreover, the upper bound shown above is further bounded independently of the dimension $d$ if $\sum_{j=1}^{\infty}\gamma_j^{1/\alpha}<\infty$ as follows: under the assumption $\sum_{j=1}^{\infty}\gamma_j^{1/\alpha}<\infty$, there exists finite $\tilde{d}$ such that
\[ \sum_{j=\tilde{d}+1}^{\infty}\gamma_j^{1/\alpha}\leq \frac{1}{2},\]
implying that $\gamma_j^{1/\alpha}\leq 1/2$ for any $j>\tilde{d}$. Therefore, for $d>\tilde{d}$, by using the elementary inequality $1+x\leq \exp(x)$, we have
\begin{align*}
    & \prod_{j=1}^{d}\left( 1+2^{2\alpha+2}\gamma_j^{1/\lambda}\zeta(\alpha/\lambda)\right)^{4\lambda-2\lambda\beta/\alpha}\max(1,2^{\beta} \gamma_j^{\beta/\alpha})\\
    & = \prod_{j=1}^{d}\left( 1+2^{2\alpha+2}\gamma_j^{1/\lambda}\zeta(\alpha/\lambda)\right)^{4\lambda-2\lambda\beta/\alpha}\times \prod_{j'=1}^{d}\max(1,2^{\beta} \gamma_{j'}^{\beta/\alpha})\\
    & \leq \prod_{j=1}^{d}\exp\left( \left(4\lambda-2\lambda\beta/\alpha\right)2^{2\alpha+2}\gamma_j^{1/\lambda}\zeta(\alpha/\lambda)\right)\times \prod_{j'=1}^{\tilde{d}}\max(1,2^{\beta} \gamma_{j'}^{\beta/\alpha})\\
    & = \exp\left( \left(4\lambda-2\lambda\beta/\alpha\right)2^{2\alpha+2}\zeta(\alpha/\lambda)\sum_{j=1}^{d}\gamma_j^{1/\lambda}\right)\times \prod_{j'=1}^{\tilde{d}}\max(1,2^{\beta} \gamma_{j'}^{\beta/\alpha})\\
    & \leq \exp\left( \left(4\lambda-2\lambda\beta/\alpha\right)2^{2\alpha+2}\zeta(\alpha/\lambda)\sum_{j=1}^{\infty}\gamma_j^{1/\alpha}\right)\times \prod_{j'=1}^{\tilde{d}}\max(1,2^{\beta} \gamma_{j'}^{\beta/\alpha}),
\end{align*}
where the right-most side is independent of the dimension $d$.

\section{Lower bound}\label{sec:lower}
Here, we prove a lower bound on the worst-case root mean squared $L_2$-approximation error $e^{\rms\text{-}L_2\text{-}\app}_{d, \alpha, \bsgamma}(A^{\rand}_{N,\bsz,\bsDelta,\Acal_d(T)})$ of our randomized lattice-based algorithm.
First, we present the following lemma, which is essentially due to Byrenheid et al.\ \cite[Lemma~4]{byrenheid2017tight}. 
This result is useful for constructing a fooling function $p_M \in H_{d,\alpha,\gamma}$ with $\Vert p_M \Vert_{d,\alpha,\gamma} = 1$, which is challenging because our algorithm involves randomly drawing $N$, $\bsz$, and $\bsDelta$.

\begin{lemma}\label{lem:lower_bound}
Given $N,d \in \NN$ such that $N$ is a prime and $d\geq 2$, let 
\begin{align*}
X_{d,\sqrt{N}} := \{-\lfloor \sqrt{N} \rfloor,\ldots, \lfloor \sqrt{N}  \rfloor\} \times \{-\lfloor \sqrt{N} \rfloor,\ldots, \lfloor \sqrt{N}  \rfloor\}\times \underbrace{\{ 0\}\times \cdots \times \{ 0\}}_{(d-2) \text{ times}}.
\end{align*}
For any $\bsz\in \{1,\ldots,N-1\}^d$, there exists $\bsh \in X_{d,\sqrt{N}}$ such that $\bsh \in P_{N,\bsz}^{\perp}$ and $h_1,h_2\neq 0$.
\end{lemma}

\begin{proof}
We prove this lemma by contradiction. Assume that $\bsh \notin P_{N,\bsz}^{\perp}$ for all $\bsh \in X_{d,\sqrt{N}}\setminus \{ \bszero\}$. From Lemma~\ref{lem:character}, this assumption implies that
\begin{align*}
\sum_{k=0}^{N-1} \exp(2\pi i \bsh \cdot \bsx_k)=0,
\end{align*}
for all $\bsh \in X_{d,\sqrt{N}}\setminus \{\bszero\}$. Let $X^{+}_{d,\sqrt{N}}:=X_{d,\sqrt{N}}\cap (\NN\cup \{0\})^d$ and consider the Fourier matrix 
\begin{align*}
M_{N,\bsz}=\left( \exp(2\pi i \bsh \cdot \bsx_k) \right)_{k\in \{ 0,1,\ldots,N-1\},\;\bsh \in X^{+}_{d,\sqrt{N}}} \in \mathbb{C}^{N\times (\lfloor \sqrt{N} \rfloor+1)^2}.
\end{align*}
We have
\begin{equation}
[M_{N,\bsz}^* M_{N,\bsz}]_{\ell,\ell'}=\sum_{k=0}^{N-1} \exp(2\pi i (\bsh^{(\ell)}-\bsh^{(\ell')})\cdot \bsx_k) =
\begin{cases}
N & \text{if } \bsh^{(\ell)}- \bsh^{(\ell')}=\bszero, \\
0 & \text{otherwise},
\end{cases}
\end{equation}
where $ \ell,\ell'\in \{1,2,\ldots,(\lfloor \sqrt{N} \rfloor+1)^2\}$.
Thus, we have
\[ M_{N,\bsz}^* M_{N,\bsz} = NI_{(\lfloor \sqrt{N} \rfloor+1)^2},\]
where $I_m$ denotes the $m\times m$ identity matrix. However, this equality means that the matrix $M_{N,\bsz}$ must have full column rank, which is not possible as we have $N< ( \lfloor \sqrt{N} \rfloor+1)^2.$ Accordingly, there must exist at least one vector $\bsh=(h_1,h_2,0,\ldots,0) \in X_{d,\sqrt{N}}\setminus \{ \bszero \}$, which depends on $\bsz$, such that
$\bsh \in P_{N,\bsz}^{\perp}$.

Now assume that there exists $\bsh\in X_{d,\sqrt{N}}\setminus \{ \bszero \}$ such that $h_2=0$ and $\bsh \in P_{N,\bsz}^{\perp}$. Under this assumption, as we have $\bsh=(h_1,0,\ldots,0)$ with $h_1\neq 0$, $\bsh \in P_{N,\bsz}^{\perp}$ implies that $h_1\cdot z_1\equiv 0 \pmod N$. However, this equality does not hold since $N$ is prime and both $h_1$ and $z_1$ are coprime to $N$. Accordingly, $h_2\neq 0$ for $\bsh \in (X_{d,\sqrt{N}}\cap P_{N,\bsz}^{\perp}) \setminus \{\bszero\}$. Similarly, it holds that $h_1\neq 0$ for $\bsh \in (X_{d,\sqrt{N}}\cap P_{N,\bsz}^{\perp}) \setminus \{\bszero\}$. Thus the result follows.
\end{proof}

\begin{remark}
	While Lemma~\ref{lem:lower_bound} is essentially due to Byrenheid et al.\ \cite[Lemma~4]{byrenheid2017tight}, the statement includes a slight improvement. 
	Byrenheid et al.\ \cite[Lemma~4]{byrenheid2017tight} would directly imply the existence of $\bsh\in X_{d,\sqrt{N}}\cap P_{N,\bsz}^{\perp}$ with $h_1\neq0$ or $h_2\neq 0$, whereas 
	Lemma~\ref{lem:lower_bound} states that such $\bsh$ exists with $h_1\neq0$ and $h_2\neq 0$. This improvement simplifies the subsequent argument in deriving the lower bound.
\end{remark}

Let us define a set $P_d(M)$ by
\begin{align*}
P_d(M):=\{\bsh=(h_1,h_2,0,\ldots,0)\in \ZZ^d, \;| h_1|,\;|h_2| \leq \lfloor \sqrt {M} \rfloor\}.
\end{align*}
Then, for given $M \in \NN$ with $M\geq 4$, consider a trigonometric function $p_{M}$ such that the Fourier coefficients are given by
$$
\widehat{p_{M}}(\bsh):=\left\{
\begin{array}{ll}
(r_{\alpha,\gamma}^2(\bsh) |P_d(M)|)^{-1/2} & \text{if } \bsh \in  P_d(M), \\
0 & \text{otherwise}. \\
\end{array}
\right.
$$
Clearly, it holds that $p_M\in H_{d,\alpha,\gamma}$ and $\| p_M\|_{d,\alpha,\gamma}=1$. With this trigonometric function, we obtain a lower bound on $e^{\rms\text{-}L_2\text{-}\app}_{d, \alpha, \bsgamma}(A^{\rand}_{N,\bsz,\bsDelta,\Acal_d})$.

\begin{theorem}
Let $M, d \in \NN, \alpha > 1/2, \bsgamma \in[0,1]^{\NN}$ and $\tau \in(0,1)$ with $M\geq 4$ and $d\geq 2$ be given. For any index set $\Acal_d$, it holds that
\begin{align*}
e^{\rms\text{-}L_2\text{-}\app}_{d, \alpha, \bsgamma}(A^{\rand}_{N,\bsz,\bsDelta,\Acal_d}) \geq \frac{\sqrt{2}\min(\gamma_1,\gamma_2)}{3M^{\alpha/2+1/2}}.
\end{align*}
\end{theorem}
\begin{proof}
From the proof of Lemma~\ref{lem:random_error_bound}, for any fixed $N\in \Pcal_M$ and $\bsz\in \Zcal_{N,d,\tau}$, it holds that
\begin{align*}
\int_{[0,1)^d}\| p_M-A^{\rand}_{N,\bsz,\bsDelta,\Acal_d}(p_M)\|_{L_2} ^2\, \mathrm{d} \bsDelta & = \sum_{\bsh\notin \Acal_d} |\widehat{p_M}(\bsh)|^2+\sum_{\bsh\in \Acal_d}\sum_{\bsl \in P_{N,\boldsymbol{z}}^\perp\setminus \{\bszero\}}|\widehat{p_M}(\bsl+\bsh)|^2\\
& \geq \sum_{\bsh\in \ZZ^d}\min\left( |\widehat{p_M}(\bsh)|^2, \sum_{\bsl \in P_{N,\boldsymbol{z}}^\perp\setminus \{\bszero\}}|\widehat{p_M}(\bsl+\bsh)|^2\right)\\
& \geq \sum_{\bsh\in P_d(M)}\min\left( |\widehat{p_M}(\bsh)|^2, \sum_{\bsl \in P_{N,\boldsymbol{z}}^\perp\setminus \{\bszero\}}|\widehat{p_M}(\bsl+\bsh)|^2\right).
\end{align*}

We know from Lemma~\ref{lem:lower_bound} that there exists a vector $
\bsl'=(\ell_1',\ell_2',0,\ldots,0)\in X_{d,\sqrt{N}}\subseteq  P_d(M)$ such that $\bsl'\in P_{N,z}^{\perp}$ and $\ell'_1,\ell'_2\neq 0$. Then, for $\bsh_1'=(-\ell_1',0,0,\ldots,0)\in P_d(M)$, we have
\begin{align*}
 \sum_{\bsl \in P_{N,\boldsymbol{z}}^\perp\setminus \{\bszero\}}|\widehat{p_M}(\bsl+\bsh_1')|^2 \geq |\widehat{p_M}(\bsl'+\bsh_1')|^2=\frac{\gamma^2_2}{\ell_2^{2\alpha}|P_d(M)|}\geq \frac{\gamma^2_2}{M^{\alpha}|P_d(M)|}.
\end{align*}
Similarly, for $\bsh_2'=(0,-\ell_2',0,\ldots,0)\in P_d(M)$, we have
\begin{align*}
 \sum_{\bsl \in P_{N,\boldsymbol{z}}^\perp\setminus \{\bszero\}}|\widehat{p_M}(\bsl+\bsh_2')|^2 \geq |\widehat{p_M}(\bsl'+\bsh_2')|^2=\frac{\gamma^2_1}{\ell_1^{2\alpha}|P_d(M)|}\geq \frac{\gamma^2_1}{M^{\alpha}|P_d(M)|}.
\end{align*}
From these, we further obtain
\begin{align*}
    & \sum_{\bsh\in P_d(M)}\min\left( |\widehat{p_M}(\bsh)|^2, \sum_{\bsl \in P_{N,\boldsymbol{z}}^\perp\setminus \{\bszero\}}|\widehat{p_M}(\bsl+\bsh)|^2\right)\\
    & \geq \min\left( |\widehat{p_M}(\bsh'_1)|^2, \sum_{\bsl \in P_{N,\boldsymbol{z}}^\perp\setminus \{\bszero\}}|\widehat{p_M}(\bsl+\bsh'_1)|^2\right)+\min\left( |\widehat{p_M}(\bsh'_2)|^2, \sum_{\bsl \in P_{N,\boldsymbol{z}}^\perp\setminus \{\bszero\}}|\widehat{p_M}(\bsl+\bsh'_2)|^2\right)\\
    & \geq \min\left( |\widehat{p_M}(\bsh'_1)|^2, \frac{\gamma^2_2}{M^{\alpha}|P_d(M)|}\right)+\min\left( |\widehat{p_M}(\bsh'_2)|^2, \frac{\gamma^2_1}{M^{\alpha}|P_d(M)|}\right)\\
    & \geq 2\min\left( \frac{\gamma^2_1}{M^{\alpha}|P_d(M)|}, \frac{\gamma^2_2}{M^{\alpha}|P_d(M)|}\right)\\
    & = \frac{2\min(\gamma_1^2,\gamma_2^2)}{M^{\alpha}|P_d(M)|}=\frac{2\min(\gamma_1^2,\gamma_2^2)}{M^{\alpha}(2\lfloor \sqrt {M} \rfloor+1)^2}\\
    & \geq \frac{2\min(\gamma_1^2,\gamma_2^2)}{M^{\alpha}(3\sqrt {M})^2}=\frac{2\min(\gamma_1^2,\gamma_2^2)}{9M^{\alpha+1}}.
\end{align*}

Since this lower bound does not depend on $N\in \Pcal_M$ nor $\bsz\in \Zcal_{N,d,\tau}$, we have
\begin{align*}
    & e^{\rms\text{-}L_2\text{-}\app}_{d, \alpha, \bsgamma}(A^{\rand}_{N,\bsz,\bsDelta,\Acal_d})\\
    & = \sup_{ \substack{f \in H_{d,\alpha,\bsgamma} \\ \|f\|_{d,\alpha,\bsgamma} \leq 1}}\sqrt{\frac{1}{|\Pcal_M|}\sum_{N \in \Pcal_M} \frac{1}{| \Zcal_{N,d,\tau}|} \sum_{\bsz \in \Zcal_{N,d,\tau}} \int_{[0,1)^d}\| f-A^{\rand}_{N,\bsz,\bsDelta,\Acal_d(T)}(f)\|_{L_2} ^2\, \mathrm{d} \bsDelta} \\
    & \geq \sqrt{\frac{1}{|\Pcal_M|}\sum_{N \in \Pcal_M} \frac{1}{| \Zcal_{N,d,\tau}|} \sum_{\bsz \in \Zcal_{N,d,\tau}} \int_{[0,1)^d}\| p_M-A^{\rand}_{N,\bsz,\bsDelta,\Acal_d}(p_M)\|_{L_2} ^2\, \mathrm{d} \bsDelta}\\
    & \geq \sqrt{\frac{1}{|\Pcal_M|}\sum_{N \in \Pcal_M} \frac{1}{| \Zcal_{N,d,\tau}|} \sum_{\bsz \in \Zcal_{N,d,\tau}} \frac{2\min(\gamma_1^2,\gamma_2^2)}{9M^{\alpha+1}}}=\frac{\sqrt{2}\min(\gamma_1,\gamma_2)}{3M^{\alpha/2+1/2}}.
\end{align*}
This completes the proof.
\end{proof}
\section{Numerical experiments}\label{sec:numerics}
Finally, we now conduct numerical experiments. 
The main purpose here is to validate our theoretical results. Since the lower bound on the $L_2$-approximation error in Section~\ref{sec:lower} is essentially derived from the two-dimensional setting, we restrict our attention to the two-dimensional case, i.e., $d=2$.

Let us consider the following two test functions:
\begin{align*}
f_1(\bsx) &=  \prod_{j=1}^d \frac{121\sqrt{33}}{100}\max \left\{\frac{25}{121}-\left(x_j-\frac{1}{2} \right)^2, 0 \right \}, \\
f_2(\bsx) &= \prod_{j=1}^{d} \left(x_j-\frac{1}{2} \right)^2\sin(2\pi x_j-\pi),
\end{align*}
with $d=2.$ The first test function is a scaled periodized kink function, which can be found in \cite{kammerer2019approximation}. A similar form of the second test function was used in \cite{goda2024randomized}. We can verify from the decay of their Fourier coefficients that $f_1\in H_{d,3/2-\varepsilon,\bsgamma}$ and $f_2\in H_{d,5/2-\varepsilon,\bsgamma}$, respectively, for arbitrarily small $\varepsilon>0$.

In the following, for both test functions, Algorithm~\ref{alg:random_cbc} is executed with fixed parameters $\alpha=2$, $\gamma_1=\gamma_2=1/3$, and $\tau=2/3$ to randomly construct good generating vectors $\bsz$. The remaining parameter, $M$, which determines the range for the number of points, is always set to a power of $2$. Similarly, we use the same parameters $\alpha$, $\gamma_1$, $\gamma_2$, and $\tau$ to generate the index set $\Acal_d(T)$, where, as motivated by the result in Corollary~\ref{cor:final_result}, we set $T=M^{2\alpha(2\alpha+1)/(4\alpha+1)}=M^{20/9}$. While the choice of parameters may be improved depending on the test functions, it is often unrealistic to assume that the smoothness and weights of the target function are known in advance. The stability of our randomized lattice-based approximation algorithm in terms of these parameters remains an open question for future work.

To approximate the exact RMSE, we generate $R=1000$ independent realizations of $A^{\rand}_{N,\bsz,\bsDelta,\Acal_d}(f)$ with given $M,\alpha,\bsgamma,\tau$, denoted by $A^{\rand}_{N^{(r)},\bsz^{(r)},\bsDelta^{(r)},\Acal_d}(f)$ for $1 \leq r \leq R$, and estimate the RMSE by
\[ \sqrt{ \frac{1}{R}\sum_{r=1}^{R}\|f-A^{\rand}_{N^{(r)},\bsz^{(r)},\bsDelta^{(r)},\Acal_d}(f)\|_{L_2}^2}. \]
Here, since the $L_2$-norm and the Fourier coefficients of $f_1$ and $f_2$ can be computed analytically, each summand is evaluated exactly as
\begin{align*}
    \|f-A^{\rand}_{N,\bsz,\bsDelta,\Acal_d}(f)\|_{L_2}^2 & =\sum_{\bsh\notin \mathcal{A}_d(T)} |\hat{f}(\bsh)|^2+\sum_{\bsh\in \Acal_d(T)}|\hat{f}(\bsh)-\hat{f}_{N,\bsz,\bsDelta}(\bsh)|^2\\
    & = \|f\|_{L_2}^2-\sum_{\bsh\in \mathcal{A}_d(T)} |\hat{f}(\bsh)|^2+\sum_{\bsh\in \Acal_d(T)}|\hat{f}(\bsh)-\hat{f}_{N,\bsz,\bsDelta}(\bsh)|^2.
\end{align*} 

We now present the results of our numerical experiments. Figure~\ref{fig:results} shows the RMSE of our randomized algorithm $A^{\rand}_{N,\bsz,\bsDelta,\Acal_d}(f)$ plotted against $M$ on a log-log scale. The left and right panels correspond to the test functions $f_1$ and $f_2$, respectively. Each panel also includes three reference lines indicating theoretical convergence rates: $M^{-\alpha/2}$, which corresponds to the known upper and lower bounds for the worst-case error of deterministic lattice-based algorithms; $M^{-\alpha(2\alpha+1)/(4\alpha+1)}$, the upper bound obtained in Section~\ref{sec:upper}; and $M^{-(\alpha/2+1/2)}$, the lower bound obtained in Section~\ref{sec:lower}. Here, $\alpha$ is not the parameter used in Algorithm~\ref{alg:random_cbc}, but instead reflects the actual smoothness of the test functions, namely $\alpha = 1.5$ for $f_1$ and $\alpha = 2.5$ for $f_2$. The thin shaded area around each RMSE curve represents the 95\% confidence interval, computed for the squared $L_2$-approximation error over $R = 1000$ independent trials for each $M$. The narrow width of the interval relative to the RMSE suggests that the reported values are statistically reliable, and the meaningful digits can be considered trustworthy.

As $M$ increases, the convergence behavior improves and appears to approach the asymptotic rate $M^{-\alpha(2\alpha+1)/(4\alpha+1)}$ shown in Section~\ref{sec:upper}. However, the observed decay of the error stays clearly above the theoretical lower bound $M^{-(\alpha/2+1/2)}$, providing numerical support for the upper bound rather than the lower bound. Closing the gap between the upper and lower bounds remains an open problem for future research.

\begin{figure}[htbp]
    \centering
    \includegraphics[width=0.45\linewidth]{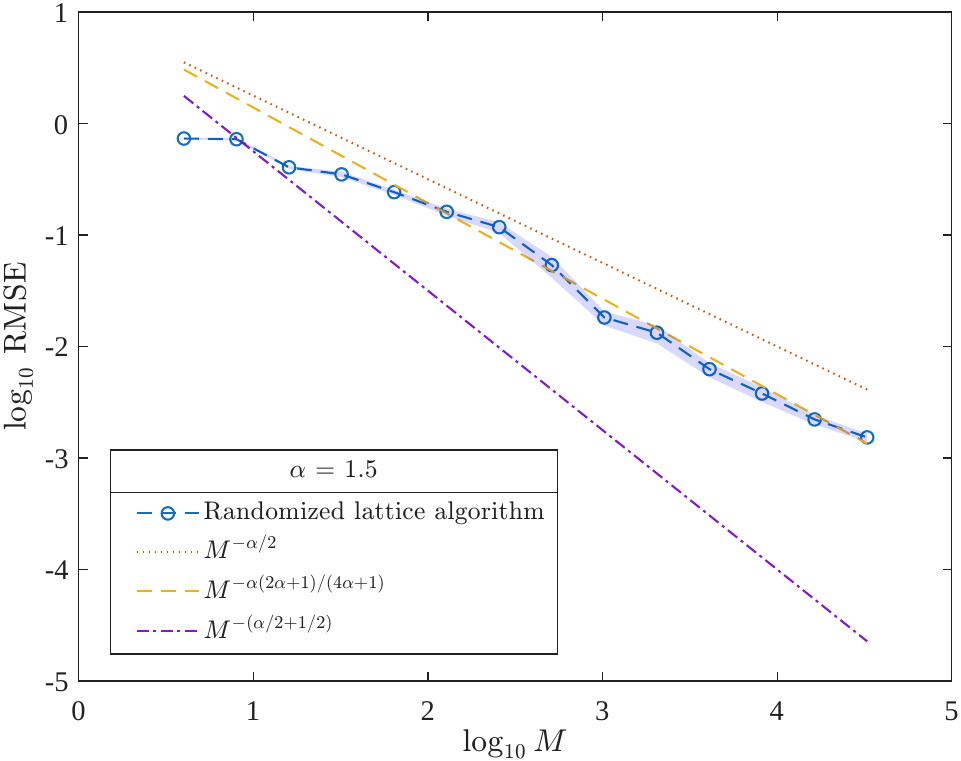}\, 
    \includegraphics[width=0.45\linewidth]{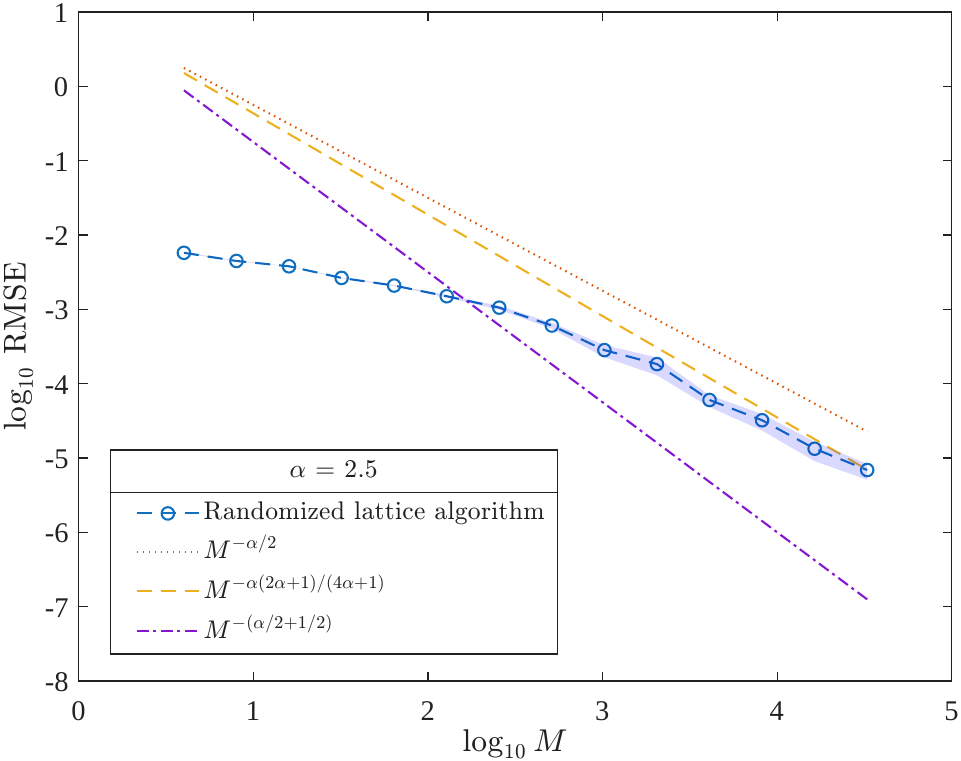}
    \caption{Results for the test functions $f_1$ (left) and $f_2$ (right).}
    \label{fig:results}
\end{figure}

\section*{Acknowledgements}
The work of T.~G.\ is supported by JSPS KAKENHI Grant Number 23K03210.

\bibliographystyle{amsplain}
\bibliography{ref.bib}

\end{document}